\documentclass[reqno]{amsart}
\usepackage{yhmath,amsmath,amsfonts,amssymb,enumerate,amsthm,appendix,caption,lipsum,}
\usepackage{enumitem}
\usepackage[english]{babel}
\usepackage{cmap,mathtools}
\usepackage{graphics} %% add this and next lines if pictures should be in esp format
\usepackage{epsfig} %For pictures: screened artwork should be set up with an 85 or 100 line screen
\usepackage{graphicx}\usepackage{epstopdf}%This is to transfer .eps figure to .pdf figure; please compile your paper using PDFLeTex or PDFTeXify.
\usepackage[pagewise]{lineno}%\linenumbers

\usepackage{a4wide}
\usepackage[margin=0.85in]{geometry}
\usepackage{cite}
\usepackage[utf8]{inputenc}
\usepackage{xcolor}
\usepackage{hyperref}
\usepackage[hyperpageref]{backref}
\hypersetup{
    colorlinks=true,
    linkcolor=myblue,
    filecolor=magenta,      
    urlcolor=mygreen,
    citecolor=mygreen,
}
\definecolor{mygreen}{rgb}{0.01,0.6,0.2}
\definecolor{myblue}{rgb}{0.01, 0.18, 1.0}
\usepackage[foot]{amsaddr}%email at bottom first page

\newtheorem{theorem}{Theorem}

\newtheorem{lemma}[theorem]{Lemma}

\theoremstyle{definition}
\newtheorem{definition}[theorem]{Definition}
\newtheorem{remark}[theorem]{Remark}

\numberwithin{equation}{section}
\numberwithin{theorem}{section}
\numberwithin{equation}{section}
\numberwithin{theorem}{section}
\title[Fractional critical systems with mixed boundary]{Fractional critical systems with mixed boundary conditions}

\author{Rohit Kumar$^{\dagger}$}
\address[R. Kumar]{Tata Institute of Fundamental Research-Centre For Applicable Mathematics, Bangalore 560065, India}
\email{rohit24@tifrbng.res.in,rohit1.iitj@gmail.com}

\author{Alejandro Ortega}
\address[A. Ortega]{Dpto. de Matem\'aticas Fundamentales, Facultad de Ciencias, UNED, 28040 Madrid, Spain}
\email{\tt alejandro.ortega@mat.uned.es}

\keywords{Fractional Laplacian, Critical Elliptic Systems, Mixed Boundary Data, Variational Methods.\\
\phantom{aa} 2010 AMS Subject Classification: Primary: 35J50, 35B33, 35R11, 35S15; Secondary: 35J61, 35Q55}
\usepackage{yhmath,amsmath,amsfonts,amssymb,enumerate,amsthm,appendix,caption,lipsum}
\usepackage[T1]{fontenc}
\usepackage{enumitem}
\usepackage{mathtools,bm}
\usepackage{cmap}
\usepackage[hyperpageref]{backref}
\usepackage{bigints}
\usepackage{esint}

\DeclarePairedDelimiter\abs{\lvert}{\rvert}
\DeclarePairedDelimiter\norm{\lVert}{\rVert}
\makeatletter
\let\oldnorm\norm
\def\norm{\@ifstar{\oldnorm}{\oldnorm*}}

\makeatletter

\DeclareMathAlphabet{\mathpzc}{T1}{pzc}{m}{it}

\definecolor{green(html/cssgreen)}{rgb}{0.0, 0.5, 0.0}

\def\R{{\mathbb R}}
\def\N{{\mathcal N}}

\def\D{{\mathcal D}}

\def\dx{{\rm d}x}

\def\dxy{{\rm d}x {\rm d}y}

\def\C{{\mathcal{C}}}
\def\Xspace{\mathcal{X}_{\Sigma_{\mathcal{D}}}^s(\C_\Omega)}
\def\Hspace{{H}_{\Sigma_{\mathcal{D}}}^s(\Omega)}

\def\del{{\partial}}

\thanks{$^{\dagger}$Corresponding author.}
\begin{document}

\begin{abstract}
In this paper, we analyze the existence of solution for a fractional elliptic system coupled by critical nonlinearities and endowed with mixed Dirichlet-Neumann boundary conditions. By means of variational methods and an orthogonalization-like process in the corresponding Sobolev space, we establish the existence of at least one weak solution.
\end{abstract}

\maketitle 
\section{Introduction}\label{S1}
In this article, we establish the existence of a non-trivial weak solution to the following system of equations coupled with critical nonlinearity and endowed with mixed boundary conditions:
\begin{equation} \label{main problem}
    \left\{
    \begin{aligned}
         (-\Delta)^{s} u&=au+bv + \alpha u|u|^{\alpha-2}|v|^\beta  & \mbox{if} ~x \in \Omega , \\[3pt]
            (-\Delta)^{s} v&=bu+cv+  \beta v|u|^{\alpha}|v|^{\beta-2} & \mbox{if} ~x \in \Omega , \\[3pt]
            u&=v=0   & \mbox{if} ~x \in \Sigma_{\mathcal{D}},\\
            \frac{\partial u}{\partial \nu}&=\frac{\partial v}{\partial \nu}=0 & \mbox{if} ~x \in \Sigma_{\mathcal{N}},
    \end{aligned}
    \right.
\end{equation}
where $\Omega \subset \R^N$ is a bounded domain with a smooth boundary, $N>2s$ and $s \in (\frac{1}{2},1)$. The operator $(-\Delta)^s$ denotes the spectral fractional Laplace operator on $\Omega$. Further, we have the following assumptions:
\begin{enumerate}
    \item $\Omega \subset \R^N$ is a bounded smooth domain;
    \item $\Sigma_{\D}$ and $\Sigma_{\N}$ are smooth $(N-1)$-dimensional submanifolds of $\del \Omega$ with $ \Sigma_{\D} \cup \Sigma_{\N}=\partial \Omega$;
    \item $\Sigma_{\D}$ is a closed manifold with positive $(N-1)$-dimensional Lebesgue measure, say $\abs{\Sigma_{\D}}=\eta \in (0, \abs{\del \Omega})$;
    \item $\Sigma_{\D} \cap \Sigma_{\N} = \emptyset$ and $\Sigma_{\D} \cap \overline{\Sigma}_{\N}=\Gamma$, with $\Gamma$ a smooth $(N-2)$-dimensional submanifold of $\del \Omega$;
    \item $a,b,c$ are real constants and $\alpha,\beta>1$ are such that $\alpha+\beta=2_s^*, 2_s^*=\frac{2N}{N-2s}, N>2s$. 
\end{enumerate}
Elliptic systems of Schrödinger-type equations or related ones have attracted a lot of attention in the last decades either for their numerous applications, for instance, in Bose-Einstein condensates (cf. \cite{Esry,Frantzeskakis, Ruegg}), in models from Nonlinear Optics (cf. \cite{Akhmediev2,Manakov}), in birefringence models (cf. \cite{Kaminow,Menyuk}), in Kerr-like photo-refractive phenomena (cf. \cite{Akhmediev, Kivshar}, in the Schrödinger-Korteweg-de Vries model for gravity-capillary waves (cf. \cite{Kakutani}) or for some plasma physics models (cf. \cite{Hojo}) or for their purely mathematical interest as they represent the vectorial case of some important scalar equations. In this regard, the system \eqref{main problem} can be seen as the vectorial mixed-boundary-data analogue of the famous Brezis-Nirenberg problem (cf. \cite{Brezis-1983}).\newline
In contrast with elliptic systems endowed either with Dirichlet or with Neumann boundary conditions, where there is a broad literature dealing with existence, multiplicity and qualitative properties of the solutions (cf. \cite{Bouchekif-2008,Chen-2012,Hsu-2011,Kang-2015,Tavares-2020,Tavares-2022} and references therein); elliptic systems with mixed boundary conditions have been much less studied and there are few works dealing with such systems (cf. \cite{Alves-2008,Morais-2012}). Thus, our main aim is to prove the existence of a non-trivial weak solution to the system \eqref{main problem} and thus extending to the fractional setting with mixed boundary conditions the results contained in \cite{alves-2000} where the authors consider \eqref{main problem} for $s=1$ and $\Sigma_{\mathcal{N}}=\emptyset$.

We will prove the existence of at least one weak solution to the system \eqref{main problem} when the eigenvalues $\mu_1$ and $\mu_2$ of the linear perturbation encoded by the matrix $A$ (see \eqref{Quadratic-form} below) lie between two consecutive eigenvalues $\Lambda_{k}<\Lambda_{k+1}$ of the spectral fractional Laplacian, namely, when $\Lambda_k\leq \mu_1\leq\mu_2<\Lambda_{k+1}$. Therefore, we will obtain the existence of solution for the so-called \textit{one sided resonant problem}. The resonant problem for the scalar equation was addressed in the classical setting for $N\geq 5$ by Capozzi, Fortunato and Palmieri (cf. \cite{Capozzi-1985} see also \cite{Zhang-1989}) and in the non-local case (for a fractional Laplacian defined through the Riesz potential) by Servadei (cf. \cite{Servadei-2014} see also \cite{Servadei-2013,Servadei-2015}).

When dealing with the system \eqref{main problem} one faces some difficulties due to the non-local nature of the operator $(-\Delta)^s$, such as, but not restricted to, obtaining explicit estimates for the action of the operator $(-\Delta)^s$ on particular functions. Precisely, in order to prove Lemma \ref{lemma3.2} below, which in turn extends \cite[Lemma 2]{Gazzola-1997} (see also \cite[Lemma 3.1]{Chabrowski-2007} and \cite[Lemma 2.1]{Morais-2012}) to the non-local mixed boundary data setting; we are forced to use the extension technique introduced by Caffarelli and Silvestre (cf. \cite{Silvestre-2007}) to reformulate the nonlocal system \eqref{main problem} in terms of a (singular) local elliptic system in the half space $\mathbb{R}^{N+1}_+$. The extension technique is also needed when one tries to construct paths for the Euler-Lagrange functional associated to \eqref{main problem} whose energy is below certain critical threshold below which the compactness holds. In this sense, the critical nature of the coupling term makes the study of system \eqref{main problem} delicate, as it involves a precise analysis of Palais-Smale sequences and a compactness-like condition that we carry out by means of the orthogonal decomposition introduced by \cite{Gazzola-1997}. However, the use of the extension technique is not fully useful at some steps of the proof and we need to work directly with the expansion in terms of the eigenfuctions to control some remainder terms arising from this orthogonal decomposition.

\textbf{Organization of the paper:} In Section \ref{Functional_Framework} we introduce the functional framework to deal with the non-local system \eqref{main problem}, paying special attention to the Sobolev-type constant associated to system \eqref{main problem}. In Section \ref{Preliminary_Results} we prove some preliminary results needed in the proof of the main result of the work. In particular, we use orthogonalization like process to obtain precise estimates that we will use to prove the main existence result of the work, which is stated and proved in Section \ref{Main_Result}. Section \ref{Geom_Compact} contains the analysis of the geometry of the Euler-Lagrange functional associated to \eqref{main problem} and the analysis of the compactness for the corresponding minimizing sequences.

%%%%%%%%%%%%%%%%%%%%%%%%%%%%%%%%%%%%%%%%%%%%%%%%%%%%%%%%%%%%%%
\section{Functional Framework}\label{Functional_Framework}
We start by recalling the definition of the spectral fractional Laplacian $(-\Delta)^s$ which is based, as its very name suggests, on the spectral decomposition of the classical Laplacian endowed with mixed boundary data. 

Let $(\lambda_j,\varphi_j)$ be the eigenvalues and eigenfunctions (normalized in $L^2(\Omega)$ norm), respectively, of the classical Laplace operator $(-\Delta)$ with homogeneous mixed Dirichlet-Neumann boundary conditions. Given a function
$$u(x)=\sum_{j\geq1}\langle u,\varphi_j\rangle_2\varphi_j(x),\quad x\in \Omega,$$
where $\langle u,v\rangle_2 =\int_{\Omega}uv\,\dx $ denotes the standard scalar product on $L^2(\Omega)$, the operator $(-\Delta)$ acts on $u$ through its action on each eigenfunction, namely,
\begin{equation*}
(-\Delta)u(x)=\sum_{j\geq 1}\lambda_j\langle u,\varphi_j\rangle_2\varphi_j(x),\quad x\in\Omega.
\end{equation*}
Following this, we define the fractional operator $(-\Delta)^s$, $0<s<1$, as the operator whose action is given by
\begin{equation}\label{spectral}
(-\Delta)^su(x)=\sum_{j\geq1}\lambda_j^s\langle u,\varphi_j\rangle_2\varphi_j(x),\quad x\in\Omega.
\end{equation}
Consequently, given two smooth functions
\begin{align*}
    u_i(x)= \sum_{j\geq1}\langle u_i, \varphi_j \rangle_2\varphi_j, \quad i=1,2,
\end{align*}
by definition we have
\begin{equation}\label{act2}
    \langle (-\Delta)^su_1,u_2\rangle_2 =\sum_{j\geq1}\lambda_j^s \langle u_1, \varphi_j \rangle_2 \langle u_2, \varphi_j \rangle_2.
\end{equation}
Equivalently, the pairs $(\Lambda_j,\varphi_j)=(\lambda_j^s,\varphi_j)$ are the eigenvalues and the eigenfunctions of $(-\Delta)^s$ respectively. 
As a consequence the spectral fractional Laplace operator $(-\Delta)^s$ is well-defined in the following Hilbert space
\begin{equation*}
    H_{\Sigma_{\mathcal{D}}}^s(\Omega):=\left\{u = \sum_{j\geq1}a_j \varphi_j \in L^2(\Omega) : ~u=0~ \text{ on } \Sigma_\mathcal{D}, \norm{u}_{H_{\Sigma_{\mathcal{D}}}^s(\Omega)}^2:= \sum_{j\geq1}a_j^2 \lambda_j^s <+\infty \right\}.
\end{equation*}
\begin{remark}
Because of \cite[Theorem 11.1]{Lions-Magenes}, if $0<s\leq \frac{1}{2}$, then we have $H_0^s(\Omega)=H^s(\Omega)$ and thus also $H_{\Sigma_{\mathcal{D}}}^s(\Omega)=H^s(\Omega)$. On the other hand, if $s\in (1/2,1)$ then we have $H_0^s(\Omega)\subsetneq H^s(\Omega)$. Therefore, the range $s\in (1/2,1)$ ensures that $H_{\Sigma_{\mathcal{D}}}^s(\Omega)\subsetneq H^s(\Omega)$.
\end{remark}
For a given $u\in H_{\Sigma_{\mathcal{D}}}^s(\Omega)$, it follows also by definition that
\begin{equation*}
    \norm{u}_{H_{\Sigma_{\mathcal{D}}}^s(\Omega)} = \norm{(-\Delta)^{\frac{s}{2}}u}_{L^2(\Omega)}.
\end{equation*}
Moreover, taking \eqref{act2} into mind, the norm $\|\cdot\|_{H_{\Sigma_{\mathcal{D}}}^s(\Omega)}$ is induced by the scalar product
\begin{equation}\label{int-b-p}
    \langle u_1,u_2\rangle_{H_{\Sigma_{\mathcal{D}}}^s(\Omega)} = \langle (-\Delta)^s u_1,u_2\rangle_{2}= \langle (-\Delta)^\frac{s}{2} u_1,(-\Delta)^\frac{s}{2}u_2\rangle_{2} = \langle  u_1,(-\Delta)^su_2\rangle_{2},
\end{equation}
for $u_1,u_2 \in H_{\Sigma_{\mathcal{D}}}^s(\Omega)$. The chain of equalities \eqref{int-b-p} can be simply stated as an
integration-by-parts like formula.

Next, we define the product space $\mathbb{H}:= H_{\Sigma_{\mathcal{D}}}^s(\Omega) \times H_{\Sigma_{\mathcal{D}}}^s(\Omega)$ with respect to the norm
\begin{align*}
    \norm{(u,v)}_\mathbb{H}^2 = \norm{u}_{H_{\Sigma_{\mathcal{D}}}^s(\Omega)}^2 + \norm{v}_{H_{\Sigma_{\mathcal{D}}}^s(\Omega)}^2 = \norm{(-\Delta)^{\frac{s}{2}}u}_{L^2(\Omega)}^2 + \norm{(-\Delta)^{\frac{s}{2}}v}_{L^2(\Omega)}^2.
\end{align*}
As \eqref{int-b-p} shows, the definition of fractional powers of the Laplace operator given by \eqref{spectral} allows us to integrate
by parts in the proper spaces, so that a natural definition of weak solution to the system \eqref{main problem} is the following.

\begin{definition}
A pair $(u,v) \in \mathbb{H}$ is said to be a weak solution to \eqref{main problem} if, for all $(\phi, \psi)\in \mathbb{H}$, we have
\begin{equation}\label{Euler Lagrange Identity}
    \begin{split}
    \int_{\Omega} (-\Delta)^{\frac{s}{2}}u (-\Delta)^{\frac{s}{2}}\phi\,\dx + \int_{\Omega} (-\Delta)^{\frac{s}{2}}v (-\Delta)^{\frac{s}{2}}\psi\,\dx=&
     \int_{\Omega} (au+bv)\phi\,\dx + \int_{\Omega} (bu+cv)\psi\,\dx\\
    & + \alpha \int_{\Omega}u|u|^{\alpha-2}|v|^{\beta}\phi\,\dx + \beta \int_{\Omega}v|u|^{\alpha}|v|^{\beta-2}\psi\,\dx.
    \end{split}
\end{equation}
\end{definition}

The energy functional $J:\mathbb{H} \rightarrow \R$ associated with \eqref{Euler Lagrange Identity} is given by
\begin{equation}\label{Energy functional}
\begin{split}
    J(u,v)=& \frac{1}{2} \int_{\Omega} |(-\Delta)^{\frac{s}{2}}u|^2\,\dx + \frac{1}{2} \int_{\Omega} |(-\Delta)^{\frac{s}{2}}v|^2\,\dx - \frac{1}{2} \int_{\Omega} (au^2+2buv+ cv^2)\,\dx -\int_{\Omega}|u|^\alpha |v|^\beta \,\dx .
\end{split}
\end{equation}

\subsection{The Extension Technique}\hfill 

Next we use the extension technique introduced by Caffarelli and Silvestre (cf. \cite{Silvestre-2007}) and also developed in \cite{Colorado-2013,Cabre-2010,Capella-2011} for problems in bounded domains, to rewrite the nonlocal elliptic system \eqref{main problem} in terms of a local (and singular) elliptic system in $\mathbb{R}^{N+1}_+$.\newline 
Associated with the domain $\Omega$, we consider the cylinder $\C_{\Omega}:= \Omega \times (0,\infty) \subset \R_+^{N+1}$. We denote with $(x,y)$ points that belong to $\C_{\Omega}$ and with $\partial_L \Omega = \partial \Omega \times (0,\infty)$ the lateral boundary of the extension cylinder. Given a function $u \in H_{\Sigma_{\mathcal{D}}}^s(\Omega)$, we define its $s$-extension $U=E_s[u]$ to the cylinder $\C_\Omega$ as the solution to the problem
\begin{align}\label{Extension problem}
    \begin{cases}
        -\text{div}(y^{1-2s} \nabla U) = 0 &\qquad \text{in } \C_\Omega\\
       \mkern+73mu B^*(U)=0  &\qquad \text{on } \partial_L\C_\Omega\\
       \mkern+68mu U(x,0)=u(x) &\qquad \text{in } \Omega \times \{y=0\},
    \end{cases}
\end{align}
where $B^*(U)= U \chi_{\Sigma^*_{\D}} + \frac{\partial U}{\partial \nu} \chi_{\Sigma^*_{\N}}$, with $\Sigma^*_{\D}= \Sigma_{\D} \times (0,\infty)$ and $\Sigma^*_{\N}= \Sigma_{\N} \times (0,\infty)$. The extension function $U$ belongs to the space
\begin{align}\label{extension space}
   \mathcal{X}_{\Sigma_{\mathcal{D}}}^s(\C_\Omega):= \overline{C_0^\infty((\Omega \cup \Sigma_\N) \times [0,\infty))}^{\norm{\cdot}_{\mathcal{X}_{\Sigma_{\mathcal{D}}}^s(\C_\Omega)}}
\end{align}
equipped with the norm
\begin{align}\label{Extension norm}
    \norm{U}_{\mathcal{X}_{\Sigma_{\mathcal{D}}}^s(\C_\Omega)}^2 = k_s \int_{\C_\Omega} y^{1-2s} |\nabla U(x,y)|^2\,\dxy,
\end{align}
with $k_s=2^{2s-1}\frac{\Gamma(s)}{\Gamma(1-s)}$. The extension operator between $\Hspace$ and $\Xspace$ is an isometry (cf. \cite{Colorado-2013}), i.e., 
\begin{align}\label{isomtery of norms}
    \norm{E_s[\varphi]}_{\Xspace} = \norm{\varphi}_{\Hspace}\qquad \text{for all } \varphi \in \Hspace.
\end{align}
The key point of the $s$-extension is that it is related to the fractional Laplacian of the original function through the formula
\begin{align}\label{relation extension function and fractional Laplace}
    \frac{\partial U}{\partial \nu^s}:= -k_s \lim\limits_{y \rightarrow 0^+} y^{1-2s}U_y(x,y) = (-\Delta)^su(x).
\end{align}
By the above arguments, we can reformulate our problem \eqref{main problem} in terms of the extension problem as follows
\begin{align}\label{main problem extension}
    \begin{cases}
        -\text{div}(y^{1-2s} \nabla U) = 0 & \quad \text{in } \C_\Omega,\\
        -\text{div}(y^{1-2s} \nabla V) = 0 & \quad \text{in } \C_\Omega,\\
        \mkern+73mu B^*(U)=0 & \quad \text{on } \partial_L\C_\Omega,\\
        \mkern+72.5mu B^*(V)=0 & \quad \text{on } \partial_L\C_\Omega,\\[3pt]
        \mkern+90mu\dfrac{\partial U}{\partial \nu^s}= aU+bV + \alpha U|U|^{\alpha-2}|V|^\beta  & \quad \text{in } \Omega \times \{y=0\},\\[7pt]
        \mkern+90mu \dfrac{\partial V}{\partial \nu^s}= bU+cV + \beta |U|^\alpha V |V|^{\beta-2} & \quad \text{in } \Omega \times \{y=0\}.
    \end{cases}
\end{align}
An energy solution to problem \eqref{main problem extension} is a function $(U,V) \in \mathbb{X}:= \Xspace \times \Xspace$ such that
\begin{align}\label{Euler Lagrange Extension problem}
\begin{split}
      &k_s \int_{\C_\Omega} y^{1-2s} \langle \nabla U, \nabla \varphi \rangle\,\dxy + k_s \int_{\C_\Omega} y^{1-2s} \langle \nabla V, \nabla \psi \rangle\,\dxy \\
    &= \int_{\Omega}\left( aU(x,0)+bV(x,0)+ \alpha U(x,0)|U(x,0)|^{\alpha-2}|V(x,0)|^\beta \right)\varphi(x,0) \,\dx  \\
    &\quad + \int_{\Omega}\left(bU(x,0)+cV(x,0)+ \beta|U(x,0)|^\alpha V(x,0) |V(x,0)|^{\beta-2}\right) \psi(x,0)\,\dx,
\end{split}
\end{align}
for all $(\varphi,\psi) \in \mathbb{X}$. Given $(U,V)\in\mathbb{X}$ a solution to \eqref{main problem extension}, the pair $(u,v) = (Tr(U), Tr(V))= (U(x,0), V(x,0))$ belongs to the space $\mathbb{H}$ and it is an energy solution to \eqref{main problem} and vice versa, if $(u,v) \in \mathbb{H}$ is a solution to \eqref{main problem}, then $(U,V)= (E_s[u], E_s[v]) \in \mathbb{X}$ is a solution to \eqref{main problem extension} and as a consequence, both formulations are equivalent. Finally, the energy functional associated to \eqref{main problem extension} is given by 
\begin{align}\label{energy functional to extension problem}
\begin{split}
   \widetilde{J}(U,V)=& \frac{k_s}{2} \int_{\C_\Omega} y^{1-2s} |\nabla U|^2\,\dxy + \frac{k_s}{2} \int_{\C_\Omega} y^{1-2s} |\nabla V|^2\,\dxy\\
   & - \frac{1}{2}\int_{\Omega}\left(aU^2(x,0)+2b U(x,0)V(x,0)+cV^2(x,0)\right)\,\dx -\int_{\Omega} |U(x,0)|^\alpha |V(x,0)|^\beta \,\dx.
\end{split}
\end{align}
Next, let us consider the fractional Sobolev constant, namely,
\begin{align}\label{Sobolev constant}
    S_{\alpha +\beta}(\Omega) = \inf\limits_{u\in \Hspace \setminus \{0\}}\frac{\norm{u}_{\Hspace}^2}{\norm{u}_{L^{\alpha+\beta}(\Omega)}^2}= \inf\limits_{U\in \Xspace \setminus \{0\}} \frac{\norm{U}_{\Xspace}^2}{\norm{U(\cdot,0)}_{L^{\alpha +\beta}(\Omega)}^2}.
\end{align}
Let us notice that the second equality of the above definition follows since $s$-extension minimizes the $\norm{\cdot}_{\Xspace}$ along all the functions with the same trace on $\{y=0\}$ (cf. \cite[Lemma 2.4]{Ortega-2019}), that is,
\begin{equation}\label{A-eq7}
    \|E_s[w(x,0)]\|^2_{\Xspace} \leq \|w\|^2_{\Xspace}, \text{ for all } w\in \Xspace.
\end{equation}

Since we consider boundary configurations such that $0<\abs{\Sigma_{\mathcal{D}}}<\abs{\partial\Omega}$, the constant $S_{\alpha +\beta}(\Omega)$ is well defined, and $S_{\alpha +\beta}(\Omega)>0$. Moreover, as $\alpha+\beta=2_s^*$, by \cite[Proposition 3.6]{Ortega-2019}, we have 
$$S_{\alpha +\beta}(\Omega)\leq 2^{\frac{-2s}{N}}S(N,s).$$
In addition, we also have the following result about the attainability of the constant $S_{\alpha +\beta}(\Omega)$.
\begin{theorem}\cite[Theorem 2.9]{Ortega-2019}\label{th_att}
If $S_{\alpha +\beta}(\Omega)<2^{\frac{-2s}{N}}S(N,s)$, then $S_{\alpha +\beta}(\Omega)$ is attained.
\end{theorem}

This result highlights one of the major differences between Dirichlet problems and mixed Dirichlet-Neumann problems. Observe that, using H\"{o}lder's inequality and the spectral definition of the fractional operator $(-\Delta)^s$, we get $S_{\alpha +\beta}(\Omega)\leq|\Omega|^{\frac{2s}{N}}\lambda_1^s(\eta)$, with $\lambda_1(\eta)$ the first eigenvalue of the classical Laplace operator endowed with mixed boundary conditions in the sets $\Sigma_{\mathcal{D}}=\Sigma_{\mathcal{D}}(\eta)$ and $\Sigma_{\mathcal{N}}= \Sigma_{\mathcal{N}}(\eta)$ being $\abs{\Sigma_{\mathcal{D}}}=\eta$. Since $\lambda_1(\eta)\to0$ as $\eta\to0^+$, (cf. \cite[Lemma 4.3]{Colorado-2003}), we have $S_{\alpha +\beta}(\Omega)\to0$ as $\eta\to0^+$.\newline 
Therefore, in contrast to the Dirichlet case ($\Sigma_{\mathcal{D}}=\partial\Omega$) where the corresponding Sobolev constant is not attained under some geometrical assumptions on $\Omega$, for instance, under star-shapeness assumptions (cf. \cite{Pohozaev-1965}); in the mixed case, the corresponding Sobolev constant
$S_{\alpha +\beta}(\Omega)$ can be achieved thanks to Theorem \ref{th_att} by taking the size of the Dirichlet boundary part small enough (cf. \cite[Theorem 1.1-(3)]{Ortega-2019}). Non-existence results based on a Poho\v{z}aev--type identity and star-shapeness like assumptions also holds for mixed problems (cf. \cite{Lions-1988,Ortega-2019}). 

As it is customary when dealing with critical problems, due to the lack of compactness of the Sobolev embedding $H_0^{s}(\Omega)  \hookrightarrow L^{p}(\Omega)$ at the critical exponent $p=2_s^*$, one first finds a threshold below of which the compactness of minimizing sequences for \eqref{main problem} holds and next continues by constructing paths whose energy is below such critical threshold and, hence, satisfying the so-called Palais-Smale condition. If the constant $S_{\alpha+\beta}(\Omega)$ is attained, it is realized by some unknown functions so we can not obtain explicit estimates that allow us to produce paths below the threshold ensuring the needed compactness. If, on the contrary, the constant $S_{\alpha+\beta}(\Omega)$ is not attained then, by Theorem \ref{th_att}, we have $S_{\alpha+\beta}(\Omega)=2^{-\frac{2s}{N}}S(N,s)$ and therefore the use of the Aubin-Talenti instantons (see \eqref{aubin} below) allows us to produce paths for the energy functional below the critical threshold (see \eqref{crit_level}) to the system \eqref{main problem}. 
Therefore, in what follows we assume that 
\begin{equation}\label{non_att_hyp}\tag{$\mathcal{H}$}
S_{\alpha+\beta}(\Omega)\ \text{is not attained, that is,}\  S_{\alpha+\beta}(\Omega)=2^{-\frac{2s}{N}}S(N,s).
\end{equation}

Further, we also define the Sobolev--type constant related to the system \eqref{main problem} as
\begin{align}\label{Sobolev type constant on product space}
     \widetilde{S}_{(\alpha,\beta)}(\Omega) =  \inf\limits_{(U,V)\in\mathbb{X} \setminus \{(0,0)\}} \frac{\norm{U}_{\Xspace}^2 + \norm{V}_{\Xspace}^2}{\displaystyle\left(\int_{\Omega}|U(x,0)|^\alpha |V(x,0)|^\beta\,\dx\right)^{\frac{2}{\alpha +\beta}}}.
\end{align}
The next result, whose proof follows as that of \cite[Theorem 5]{alves-2000}, establishes a fundamental relation between the Sobolev constants \eqref{Sobolev constant} and \eqref{Sobolev type constant on product space}. Precisely, it allows us to relate the functions where the Sobolev constant \eqref{Sobolev constant} is attained with the pair where \eqref{Sobolev type constant on product space} is attained so we can relate the energy of the paths constructed by using the Aubin-Talenti instantons with the critical threshold below of which the compactness holds.
\begin{theorem}\label{important relation}
    Let $\Omega$ be a domain (need not be bounded) and $\alpha + \beta \leq 2_s^*$. Then we have
    \begin{align}\label{equivalence in the Sobolev constant}
    \widetilde{S}_{(\alpha,\beta)}(\Omega) = \left[ \left(\frac{\alpha}{\beta} \right)^{\frac{ \beta}{\alpha +\beta}} + \left(\frac{\alpha}{\beta} \right)^{-\frac{ \alpha}{\alpha +\beta}} \right] S_{\alpha +\beta}(\Omega).
\end{align}
Moreover, if $S_{\alpha +\beta}(\Omega)$ is attained at $u$, then $\widetilde{S}_{(\alpha,\beta)}(\Omega)$ is attained at the pair $(\mathcal{A}u,\mathcal{B}u)$ for any real constants $\mathcal{A}$ and $\mathcal{B}$ such that $\frac{\mathcal{A}}{\mathcal{B}}=\sqrt{\frac{\alpha}{\beta}}$.
\end{theorem}
\begin{proof}
Let $\{\Psi_n\} \subset \Xspace$ be a minimizing sequence for \eqref{Sobolev constant}. Let $\gamma,\delta >0$ to be chosen later. Taking $U_n = \gamma \Psi_n$ and $V_n=\delta \Psi_n$ in the quotient \eqref{Sobolev type constant on product space}, we have that
\begin{align*}
    \frac{\norm{U_n}_{\Xspace}^2 + \norm{V_n}_{\Xspace}^2}{\left(\int_{\Omega}|U_n(x,0)|^\alpha |V_n(x,0)|^\beta\,\dx\right)^{\frac{2}{\alpha +\beta}}} \geq \widetilde{S}_{(\alpha,\beta)}(\Omega).
\end{align*}
From the above inequality, we deduce that
\begin{align}\label{eq1}
    \frac{(\gamma^2+\delta^2)\norm{\Psi_n}_{\Xspace}^2}{(\gamma^\alpha \delta^\beta)^{\frac{2}{\alpha +\beta}}\left(\int_{\Omega}|\Psi_n(x,0)|^{\alpha+\beta} \,\dx\right)^{\frac{2}{\alpha +\beta}}} \geq \widetilde{S}_{(\alpha,\beta)}(\Omega).
\end{align}
Since 
$
    \frac{\gamma^2+\delta^2}{(\gamma^\alpha \delta^\beta)^{\frac{2}{\alpha +\beta}}}= \left(\frac{\gamma}{\delta} \right)^{\frac{2 \beta}{\alpha +\beta}} + \left(\frac{\gamma}{\delta} \right)^{-\frac{2 \alpha}{\alpha +\beta}},
$
let us consider the differentiable map $g:(0,\infty) \rightarrow \R$ given by
\begin{align}\label{map g}
    g(t)= t^{\frac{2 \beta}{\alpha +\beta}} + t^{-\frac{2 \alpha}{\alpha +\beta}}.
\end{align}
As $g$ is a continuous function and $\lim\limits_{t \rightarrow +\infty} g(t)= \lim\limits_{t \rightarrow 0+} g(t) =+\infty$, it has a global minimum and one can verify that $t_1=\sqrt{\frac{\alpha}{\beta}}$ is a point of global minima and moreover
\begin{align}\label{g at minima point}
    g\left(\sqrt{\frac{\alpha}{\beta}}\right) =  \left(\frac{\alpha}{\beta} \right)^{\frac{ \beta}{\alpha +\beta}} + \left(\frac{\alpha}{\beta} \right)^{-\frac{ \alpha}{\alpha +\beta}}.
\end{align}
Choosing $\gamma$ and $\delta$ in \eqref{eq1} such that $\frac{\gamma}{\delta}=\sqrt{\frac{\alpha}{\beta}}$ we have that
\begin{align}\label{eq2}
    \left[\left(\frac{\alpha}{\beta} \right)^{\frac{ \beta}{\alpha +\beta}} + \left(\frac{\alpha}{\beta} \right)^{-\frac{ \alpha}{\alpha +\beta}} \right] S_{\alpha +\beta}(\Omega) \geq \widetilde{S}_{(\alpha,\beta)}(\Omega).
\end{align}
Now we need to prove the reverse inequality in order to complete the proof. Assume that $(U_n,V_n) \in \mathbb{X}$ be a minimizing sequence for \eqref{Sobolev type constant on product space}. Define $\Phi_n = t_n V_n \in \Xspace$ for some $t_n >0$ such that
\begin{align}\label{eq3}
    \int_{\Omega}|U_n(x,0)|^{\alpha+\beta}\,\dx = \int_{\Omega}|\Phi_n(x,0)|^{\alpha+\beta}\,\dx. 
\end{align}
The Young's inequality yields
\begin{align}\label{eq4}
    \int_{\Omega}|U_n(x,0)|^{\alpha} |\Phi_n(x,0)|^{\beta}\,\dx \leq \frac{\alpha}{\alpha +\beta} \int_{\Omega}|U_n(x,0)|^{\alpha+\beta} \,\dx  + \frac{\beta}{\alpha +\beta} \int_{\Omega}|\Phi_n(x,0)|^{\alpha+\beta} \,\dx.
\end{align}
In this way using \eqref{eq4} we deduce that
\begin{align*}
    \frac{\norm{U_n}_{\Xspace}^2 + \norm{V_n}_{\Xspace}^2}{\displaystyle\left(\int_{\Omega}|U_n(x,0)|^\alpha |V_n(x,0)|^\beta\,\dx\right)^{\frac{2}{\alpha +\beta}}} &= \frac{t_n^{\frac{2\beta}{\alpha +\beta}}\left(\norm{U_n}_{\Xspace}^2 + \norm{V_n}_{\Xspace}^2 \right)}{\displaystyle\left(\int_{\Omega}|U_n(x,0)|^\alpha |\Phi_n(x,0)|^\beta\,\dx\right)^{\frac{2}{\alpha +\beta}}} \\
    &\geq \frac{t_n^{\frac{2\beta}{\alpha +\beta}} \norm{U_n}_{\Xspace}^2 }{\displaystyle\left(\int_{\Omega}|U_n(x,0)|^{\alpha +\beta}\,\dx\right)^{\frac{2}{\alpha +\beta}}} + \frac{t_n^{\frac{2\beta}{\alpha +\beta}} t_n^{-2} \norm{\Phi_n}_{\Xspace}^2 }{\displaystyle\left(\int_{\Omega}|\Phi_n(x,0)|^{\alpha +\beta}\,\dx\right)^{\frac{2}{\alpha +\beta}}}\\
    & \geq t_n^{\frac{2\beta}{\alpha +\beta}} S_{\alpha +\beta}(\Omega) + t_n^{\frac{- 2\alpha}{\alpha +\beta}} S_{\alpha +\beta}(\Omega)\\
    &= g(t_n) S_{\alpha +\beta}(\Omega)\\
    & \geq g\left(\sqrt{\frac{\alpha}{\beta}} \right) S_{\alpha +\beta}(\Omega).
\end{align*}
Passing to the limit as $n \rightarrow \infty$ in the above inequality we conclude
\begin{align}\label{eq5}
    \widetilde{S}_{(\alpha,\beta)}(\Omega) \geq \left[ \left(\frac{\alpha}{\beta} \right)^{\frac{ \beta}{\alpha +\beta}} + \left(\frac{\alpha}{\beta} \right)^{-\frac{ \alpha}{\alpha +\beta}} \right] S_{\alpha +\beta}(\Omega).
\end{align}
\end{proof}
\section{Preliminary Results}\label{Preliminary_Results}
Let us introduce some notation used in the sequel. Given a set $\mathcal{M}\subset\mathbb{R}^{N+1}_+$, we denote by $|\mathcal{M}|_s$ the measure induced by the weight $y^{1-2s}$ belonging to the Muckenhoupt class $A_2$, namely,
$$|\mathcal{M}|_s:= \kappa_s \int_{\mathcal{M}}y^{1-2s}\, \dxy.$$ 
Next, given $X_0 = (x_0,0) \in \R^{N+1}$, let $B_\rho^+(X_0):= B_\rho(X_0) \cap \R_+^{N+1}$ be the upper ball of radius $\rho>0$ centered at $X_0$, where $B_\rho(X_0)$ denotes the usual $(N+1)$-dimensional ball. Notice that
\begin{equation}\label{scaling}
    |B_\rho^+(X_0)|_s = \rho^{N+2(1-s)} |B_1^+(X_0)|_s.
\end{equation}
Assume without loss of generality that $0 \in \Sigma_\N$. For every $m \in \mathbb{N}$, we define the function
\begin{align}\label{test function em}
    \mathcal{E}_m(x) = \begin{cases}
        0, \quad &\text{ if } x \in \Omega \cap B_{\frac{1}{m}},\\
        m|x|-1, \quad &\text{ if } x \in \Omega \cap \left(B_{\frac{2}{m}} \setminus B_{\frac{1}{m}} \right),\\
        1, \quad &\text{ if } x \in \Omega \setminus B_{\frac{2}{m}},
    \end{cases}
\end{align}
where $B_r$ is an open ball with center at origin and radius $r$. Further, we consider the approximate eigenfunction $$\varphi_i^m := \mathcal{E}_m \varphi_i.$$ 
Let us also define the projection space
\begin{align}\label{H-minus-space}
     \mathbb{H}_{k}^- := \text{span}\{\varphi_i : i=1,2,...,k\} \quad \text{ and } \quad \mathbb{H}_{k}^+ = (\mathbb{H}_{k}^-)^\perp,
\end{align}
and the approximation space
\begin{align}\label{approximate-H-minus-space}
    \mathbb{H}_{m,k}^- := \text{span}\{\varphi_i^m : i=1,2,...,k\}.
\end{align}
Finally, let us denote by $\mathcal{S}:= \left\{u\in \Hspace : \|u\|_{L^2(\Omega)}=1\right\}$ the unit sphere in $L^2(\Omega)$.

\begin{lemma}\label{orthogonality of approximate eigenfunctions}
    If $\lambda_i^s$ and $\lambda_j^s$ are two distinct eigenvalues of spectral fractional Laplacian $(-\Delta)^s, s\in \left(\frac{1}2{},1\right),$ with homogeneous mixed boundary conditions and $\varphi_i,\varphi_j$ are the corresponding eigenfunctions, then we have
    \begin{itemize}
        \item[\rm(a)] $\lim\limits_{m \rightarrow \infty} \int_{\Omega} \varphi_i^m \varphi_j^m\,\dx=0$ ,
        \item[\rm(b)] $\lim\limits_{m \rightarrow \infty} \int_{\Omega} (-\Delta)^{\frac{s}{2}}\varphi_i^m (-\Delta)^{\frac{s}{2}}\varphi_j^m\,\dx=0$.
    \end{itemize}
\end{lemma}
\begin{proof}
    Since $0\leq \mathcal{E}_m(x)\leq 1$ and $\mathcal{E}_m(x) \rightarrow 1$ a.e. in $\Omega$ as $m \rightarrow \infty$, the proof of $\rm(a)$ follows immediately by the dominated convergence theorem. To prove $\rm(b)$, observe that
\begin{equation}\label{l1-eq1}
    \begin{split}
        \int_{\Omega} (-\Delta)^{\frac{s}{2}}\varphi_i^m (-\Delta)^{\frac{s}{2}}\varphi_j^m\,\dx=&\langle\varphi_i^m ,\varphi_j^m \rangle_{\Hspace}\\
        =&\langle(\varphi_i^m-\varphi_i)+\varphi_i ,(\varphi_j^m-\varphi_j)+\varphi_j \rangle_{\Hspace}\\
        =&\langle\varphi_i^m-\varphi_i ,\varphi_j^m-\varphi_j \rangle_{\Hspace}+\langle\varphi_i^m-\varphi_i,\varphi_j \rangle_{\Hspace}+\langle\varphi_i ,\varphi_j^m-\varphi_j\rangle_{\Hspace}\\
        \leq&\norm{\varphi_i^m-\varphi_i}_{\Hspace}\left(\norm{\varphi_j^m-\varphi_j}_{\Hspace}+\norm{\varphi_j}_{\Hspace}\right)\\
        &+\norm{\varphi_j^m-\varphi_j}_{\Hspace} \norm{\varphi_i}_{\Hspace}.
    \end{split}
    \end{equation}
Since, by Lemma \ref{lemma3.2} below, we have $\varphi_k^m \rightarrow \varphi_k$ in $\Hspace$ as $m \rightarrow \infty$ for any $k\geq1$, the results follows.  
\end{proof}
As commented before, due to the non-local nature of the operator $(-\Delta)^s$ it is not possible to obtain explicit expressions for the action of $(-\Delta)^s$ on a given function. In particular, we are interested on the action of $(-\Delta)^s$ on some truncation functions as those of the family $\mathcal{E}_m$. In this step we are thus required to make full use of the extension technique. Therefore, let us define the function
\begin{align}\label{Extension-T-rho}
    \mathcal{T}_m(x,y) = \begin{cases}
        0 \quad &\text{ if } (x,y) \in \C_\Omega \cap B_{\frac{1}{m}}^+(0,0),\\
        m|(x,y)|-1 \quad &\text{ if } (x,y) \in \C_\Omega \cap A_m, \text{ where } A_m:=B_{\frac{2}{m}}^+(0,0) \setminus B_{\frac{1}{m}}^+(0,0) ,\\
        1 \quad &\text{ if } (x,y) \in \C_\Omega \setminus B_{\frac{2}{m}}^+(0,0),
    \end{cases}
\end{align}
where $|(x,y)|=\left(|x|^2 + y^2 \right)^{\frac{1}{2}}, x\in \Omega, y \in (0,\infty)$. Let us stress that $\mathcal{T}_m(x,0)=\mathcal{E}_m(x)$ but $\mathcal{T}_m(x,y)\neq E_s[\mathcal{E}_m(x)]$. Because of \eqref{scaling}, we have 
\begin{align}\label{A-eq2}
\begin{split}
    \|\mathcal{T}_m(x,y)\|_{\Xspace}^2 &= \kappa_s \int_{\C_\Omega} y^{1-2s} |\nabla \mathcal{T}_m|^2\,\dxy= m^2 \kappa_s \int_{\mathcal{C}_\Omega \cap A_m} y^{1-2s}\,\dxy\\
    &\leq  ~m^2 \left[|B_{\frac{2}{m}}^+(0,0)|_s - |B_{\frac{1}{m}}^+(0,0)|_s\right]\\
    &= \left( 2^{N+2(1-s)} - 1\right)|B_{1}^+(0,0)|_s m^{2s-N}.
\end{split}
\end{align}
Next we extend \cite[Lemma 2]{Gazzola-1997} (see also \cite[Lemma 2.1]{Morais-2012}) to the fractional mixed-boundary-data setting.
\begin{lemma}\label{lemma3.2}
    The approximate eigenfunction $\varphi_i^m \rightarrow \varphi_i$ in $\Hspace$ as $m \rightarrow \infty$ and 
    \begin{align}
       \max\limits_{u \in \mathbb{H}^{-}_{m,k} \cap \mathcal{S}} \|u\|_{\Hspace}^2 \leq \lambda_k^s + C m^{2s-N},
    \end{align}
   where $C=C(N,s,k)>0$ is independent of $m$.
\end{lemma}
\begin{proof}

Let $\Phi_i(x,y)=E_s[\varphi_i]$, so that $\Phi_i(x,0)=\varphi_i(x)$. Then, by \eqref{isomtery of norms} and \eqref{A-eq7}, we have
\begin{equation}\label{A-eq1}
\begin{split}
   \|\varphi_i^m- \varphi_i\|_{\Hspace}^2&= \|E_s[\varphi_i^m-\varphi_i]\|_{\Xspace}^2=\|E_s[(\mathcal{E}_m-1)\varphi_i]\|_{\Xspace}^2\leq \|(\mathcal{T}_m-1)\Phi_i\|_{\Xspace}^2.
   \end{split}
\end{equation}
Now we estimate the norm on the right hand side in the above inequality. Let us set $w_m=\mathcal{T}_m(x,y)-1$, then 
\begin{align}\label{A-eq3}
\begin{split}
     \|w_m \Phi_i\|_{\Xspace}^2=& \kappa_s \int_{\C_\Omega} y^{1-2s} |\nabla (w_m \Phi_i)|^2\,\dxy\\
    =&\kappa_s \int_{\C_\Omega} y^{1-2s} \Phi_i^2 |\nabla w_m|^2 \,\dxy + \kappa_s \int_{\C_\Omega} y^{1-2s} w_m^2|\nabla \Phi_i|^2 \,\dxy  \\
    &+ 2 \kappa_s \int_{\C_\Omega} y^{1-2s} \langle w_m\nabla \Phi_i, \Phi_i \nabla w_m \rangle \,\dxy .\\
    =& I_1+I_2+I_3.
    \end{split}
\end{align}
Since, by \cite[Proposition 7]{Ortega-2024-Subcritical} we have $\varphi_i\in L^{\infty}(\Omega)$, by \cite{Ortega-2021-Reg} we have $\Phi_i\in L_{loc}^\infty(\C_\Omega)$. Then, as in \eqref{A-eq2}, we get
\begin{align}\label{I1-estimate}
    \begin{split}
        I_1&= \kappa_s \int_{\C_\Omega} y^{1-2s} \Phi_i^2 |\nabla w_m|^2 \,\dxy \leq \|\Phi_i\|_{L^\infty(A_m)}^2 \|w_m\|_{\Xspace}^2\\
        & \leq C_1(N,s,i) m^{2s-N}.
    \end{split}
\end{align}
{Since the function $w_m$ is supported at the Neumann boundary, by interior elliptic regularity (cf. \cite{Cabre-2010,Capella-2011})}, the estimate on $I_2$ is given by
\begin{align}\label{I2-estimate}
    \begin{split}
        I_2&= \kappa_s \int_{\C_\Omega} y^{1-2s} w_m^2|\nabla \Phi_i|^2 \,\dxy\\
        & \leq \|\nabla\Phi_i\|_{{L^\infty(B_{\frac{2}{m}}^+(0,0))}}^2 \int_{\C_\Omega} y^{1-2s} w_m^2\,\dxy\\
        &= \|\nabla\Phi_i\|_{{L^\infty(B_{\frac{2}{m}}^+(0,0))}}^2 \left(\int_{\mathcal{C}_\Omega \cap B_{\frac{1}{m}}^+(0,0)} y^{1-2s}\,\dxy +\int_{\mathcal{C}_\Omega \cap A_m} \left(2 - m|(x,y)| \right)^2 y^{1-2s}\,\dxy \right)\\
        &\leq \|\nabla\Phi_i\|_{{L^\infty(B_{\frac{2}{m}}^+(0,0))}}^2 \left(\int_{\mathcal{C}_\Omega \cap} B_{\frac{1}{m}}^+(0,0) y^{1-2s}\,\dxy +\int_{\mathcal{C}_\Omega \cap A_m} y^{1-2s}\,\dxy \right)\\
        &= \|\nabla\Phi_i\|_{{L^\infty(B_{\frac{2}{m}}^+(0,0))}}^2 \left(\int_{\mathcal{C}_\Omega \cap B_{\frac{2}{m}}^+(0,0)} y^{1-2s}\,\dxy \right)\\
        & \leq C_2(N,s,i)m^{2s-N-2} .
        \end{split}
\end{align}
Finally, using the Cauchy-Schwarz inequality and the fact that $0\leq |w_m(x,y)|\leq 1$, we get
\begin{align}\label{I3-estimate}
  \begin{split}
        I_3&= 2 \kappa_s \int_{\C_\Omega} y^{1-2s} \langle w_m\nabla \Phi_i, \Phi_i \nabla w_m \rangle \,\dxy \leq 2 \kappa_s \int_{\C_\Omega} y^{1-2s}  |w_m\nabla \Phi_i| |\Phi_i \nabla w_m | \,\dxy\\
        & \leq 2 \kappa_s  \|\nabla \Phi_i\|_{{L^\infty(A_m)}} \|\Phi_i\|_{{L^\infty(A_m)}}\int_{\mathcal{C}_\Omega \cap A_m} y^{1-2s}  |\nabla w_m | \,\dxy\\
        & = 2 \kappa_s  \|\nabla \Phi_i\|_{{L^\infty(A_m)}} \|\Phi_i\|_{{L^\infty(A_m)}} m \int_{\mathcal{C}_\Omega \cap A_m} y^{1-2s}\,\dxy\\
        &\leq C_3(N,s,i) m^{2s-N-1}.
        \end{split}
\end{align}
From \eqref{A-eq1}, \eqref{A-eq3}, \eqref{I1-estimate},\eqref{I2-estimate} and \eqref{I3-estimate}, we deduce that
\begin{align}\label{A-eq4}
   \|\varphi_i^m- \varphi_i\|_{\Hspace}^2&= \|E_s[\varphi_i^m-\varphi_i]\|_{\Xspace}^2 \notag\\
   & \leq C_1(N,s,i) m^{2s-N} + C_2(N,s,i) m^{2s-N-2} + C_3(N,s,i)m^{2s-N-1}.
\end{align}
Taking $m \rightarrow \infty$ in the above inequality, we get $\varphi_i^m \rightarrow \varphi_i$ in $\Hspace$ for each $i=1,2,...,k$.

Next, let us set $v\in \mathbb{H}_k^- \cap \mathcal{S}$, that is, $v = \sum_{i=1}^{k}\alpha_i\varphi_i$ with $\sum_{i=1}^{k}\alpha_i^2=1$. Then, we have $$v_m:= \mathcal{E}_m v = \sum_{i=1}^{k}\alpha_i\mathcal{E}_m \varphi_i = \sum_{i=1}^{k}\alpha_i \varphi_i^m,$$ which implies $v_m \in \mathbb{H}_{m,k}^-$. For each $v \in \mathbb{H}_k^- \cap \mathcal{S}$, let us set $a_m(v) := \|v_m\|_2^{-1}$, so that $a_m(v)v_m \in \mathbb{H}_{m,k}^- \cap S$. Now, we have
\begin{align*}
    1 = \int_{\Omega} a_m^2(v) v_m^2 = a_m^2(v) \left( 1 - \int_{\Omega \cap B_{\frac{2}{m}}} v^2 + \int_{\Omega \cap \left(B_{\frac{2}{m}} \setminus B_{\frac{1}{m}}\right)} \mathcal{E}_m^2 v^2 \right) \begin{cases}
        \leq a_m^2(v),\\
        \geq a_m^2(v)(1 - C\|v\|_{\infty}^{2} m^{-N}).
    \end{cases}
\end{align*}
Thus, we get
\[
1 \leq a_m^2(v) \leq 1 + C \|v\|_{\infty}^{2} m^{-N}.
\]
Let now $z_m \in \mathbb{H}^{-}_{m,k} \cap \mathcal{S}$ such that $\|z_m\|_{\Hspace}^2 = \max\limits_{\Hat{z} \in \mathbb{H}^{-}_{m,k} \cap \mathcal{S}} \|\Hat{z}\|_{\Hspace}^2$. Then
\[
z_m = a_m(z) \sum_i \beta_i \varphi_i^m = a_m(z) \mathcal{E}_m \sum_i \beta_i \varphi_i = a_m(z) \mathcal{E}_m z.
\]
Let $Z(x,y)=E_s[z]$. Then we have
\begin{equation}\label{A-eq5}
\begin{split}
    \|z_m\|_{\Hspace}^2 &=a_m^2(z)\|\mathcal{E}_m z\|_{\Hspace}^2 = a_m^2(z)\|E_s[\mathcal{T}_m(x,0) Z(x,0)]\|_{\Xspace}^2  \\
   & \leq a_m^2(z)\|\mathcal{T}_{m}(x,y) Z(x,y)\|_{\Xspace}^2  \\
    &\leq (1 + C \|z\|_{\infty}^{2} m^{-N})\|\mathcal{T}_{m}(x,y) Z(x,y)\|_{\Xspace}^2.
\end{split}
\end{equation}
Using \eqref{A-eq2} and the fact that $0\leq \mathcal{T}_m \leq 1$, we estimate the integral in the last inequality as follows
\begin{equation}\label{A-eq6}
\begin{split}
\|\mathcal{T}_{m}Z\|_{\Xspace}^2=&\kappa_s\int_{\C_\Omega} y^{1-2s} \left|\nabla (\mathcal{T}_mZ)\right|^2\,\dxy \\
    %& =  \kappa_s\int_{\C_\Omega} y^{1-2s} \left|\mathcal{T}_m \nabla Z + Z \nabla \mathcal{T}_m\right|^2 \,\dxy \\
    =& \kappa_s\int_{\C_\Omega} y^{1-2s} \left( \left|\mathcal{T}_m \nabla Z\right|^2 + \left|Z \nabla \mathcal{T}_m\right|^2 + 2 \langle \mathcal{T}_m \nabla Z, Z \nabla \mathcal{T}_m \rangle \right)\,\dxy \\
    \leq&  \kappa_s\int_{\C_\Omega} y^{1-2s}  | \nabla Z|^2\,\dxy + \kappa_s \|Z\|_\infty^2 \int_{\C_\Omega} y^{1-2s} | \nabla \mathcal{T}_m|^2\,\dxy \\
    & + 2\kappa_s\|Z\|_\infty \|\nabla Z\|_\infty \int_{\C_\Omega} y^{1-2s} | \nabla \mathcal{T}_m| \,\dxy \\
    =& \|E_s[z]\|_{\Xspace}^2 +\|Z\|_\infty^2 \|\mathcal{T}_m\|_{\Xspace}^2\\
    &+2\kappa_s  \|Z\|_\infty \|\nabla Z\|_\infty m \int_{\mathcal{C}_\Omega \cap A_m} y^{1-2s}\,\dxy \\
    \leq& \|z\|_{\Hspace}^2 + C_4(N,s,k) m^{2s-N} +C_5(N,s,k) m^{2s-N-1} \\
    \leq& \lambda_k^s + C_6(N,s,k) m^{2s-N}.
    \end{split}
\end{equation}
From \eqref{A-eq5} and \eqref{A-eq6}, we finally get $\|z_m\|_{\Hspace}^2 \leq \lambda_k^s + C_6(N,s,k) m^{2s-N}$, which completes the proof. 
\end{proof}
\subsection{Estimates for the Sobolev extremal functions}\hfill

Next we collect some useful estimates for the family of extremal functions of the fractional Sobolev inequality, which is given by
\begin{equation}\label{aubin}
    u_\varepsilon (x) = \frac{\varepsilon^{\frac{N-2s}{2}}}{(\varepsilon^2 + |x|^2)^{\frac{N-2s}{2}}},\quad x \in \R^N \text{ and } \varepsilon>0.
\end{equation}
Let us denote by $U_\varepsilon = E_s[u_\varepsilon]$ its $s$-extension. It is worth to mentioning that both the functions $u_\varepsilon$ and its $s$-extension $U_\varepsilon$, are self-similar functions (cf. \cite{Barrios-2012}), namely
\begin{equation}\label{eq:3.1}
    u_\varepsilon(x) = \varepsilon^{-\frac{N-2s}{2}} u_1\left(\frac{x}{\varepsilon}\right) \quad \text{and} \quad U_\varepsilon (x,y) = \varepsilon^{-\frac{N-2s}{2}} U_1 \bigg( \frac{x}{\varepsilon}, \frac{y}{\varepsilon} \bigg).
\end{equation}
Consider a smooth non-increasing cut-off function $\phi_0 (t) \in C^\infty(\mathbb{R}_+)$, satisfying $\phi_0 (t) = 1$ for $0 \leq t \leq \frac{1}{4}$ and $\phi_0 (t) = 0$ for $t \geq \frac{1}{2}$, and $|\phi_0' (t)| \leq C$ for any $t \geq 0$. 

Next, we construct a cut-off function $\phi(x,y)$ centered at a Neumann boundary point $x_0\in\Sigma_{\mathcal{N}}$. In particular, take $x_0\in\Sigma_{\mathcal{N}}$ and define, for some $\rho > 0$ small enough such that $B_\rho(x_0)\cap\Sigma_{\mathcal{D}}=\emptyset$, the function $\phi(x,y)=\phi_{\rho}(x,y)=\phi_0(\frac{r_{xy}}{\rho})$ with $r_{xy}=|(x-x_0,y)|=(|x-x_0|^2+y^2)^{\frac{1}{2}}$. For our further reference, we assume that $x_0=0 \in \Sigma_\mathcal{N}$. Then, the following holds.
\begin{lemma}\rm{(\cite[Lemma 12]{Ortega2023})}\label{lemma-Ortega2019}
The family $\{\phi_{\rho} U_{\varepsilon}\}$ and its trace on $\{y=0\}$, namely $\{\phi_{\rho} u_{\varepsilon}\}$, satisfy
\begin{equation}\label{eq:3.2}
\|\phi_\rho U_{\varepsilon}\|_{\mathcal{X}_{\Sigma_{\mathcal{D}}}^s(\mathcal{C}_{\Omega})}^2=\frac{1}{2}\|U_{\varepsilon}\|_{\mathcal{X}_{\Sigma_{\mathcal{D}}}^s(\mathcal{C}_{\Omega})}^2+ O\left(\left(\frac{\varepsilon}{\rho}\right)^{N-2s}\right),
\end{equation}
and
\begin{equation}\label{eq:3.3}
\|\phi_\rho u_{\varepsilon}\|_{L^{2_s^*}(\Omega)}^{2_s^*}=\frac{1}{2}\|u_{\varepsilon}\|_{L^{2_s^*}(\mathbb{R}^N)}^{2_s^*}+O\left(\left(\frac{\varepsilon}{\rho}\right)^N\right).
\end{equation}
for $\varepsilon>0$ small enough.
\end{lemma}

The above estimates will be crucial in order to construct paths for the energy functional associated to the system \eqref{main problem} below a critical threshold for which the compactness holds that we compute in next section.
\section{Geometry and Compactness}\label{Geom_Compact}
First we recall the following Rabinowitz Linking theorem \cite[Theorem 5.3]{Rabinowitz-1986-book}.
\begin{theorem}[Rabinowitz Linking Theorem]
    Let $E$ be a real Banach space with $E= V \oplus X$, where $V$ is finite dimensional. Suppose $J \in C^{1}(E,\R)$ satisfies Palais-Smale condition (in short (PS) condition) and
    \begin{enumerate}\rm
        \item there are constants $\rho,\sigma>0$ such that $J|_{\partial B_\rho \cap X} \geq \sigma$ and
        \item there is an $e \in \partial B_1 \cap X$ and $R> \sigma$ such that if $Q \equiv (\overline{B}_R \cap V) \oplus \{re: 0<r<R\}$, then $J|_{\partial Q} \leq 0$.
    \end{enumerate}
Then $J$ possesses a critical value $\widetilde{c} \geq \sigma$ which can be characterized as $\widetilde{c} = \inf\limits_{h \in \mathcal{H}} \sup\limits_{u \in Q} J(h(u))$, where 
\begin{align*}
    \mathcal{H} = \{ h : \overline{Q} \to E \text{ continuous}; \ h|_{\partial Q} = \text{id}\}.
\end{align*}
In the above, $\partial Q$ refers to the boundary of $Q$ relative to $V \oplus span\{e\}$.
\end{theorem}
Next, for \( m,k \in \mathbb{N} \) and $\varrho,\varepsilon>0$, we define
\[
Q_m^\varepsilon = [0,\varrho] {\bf U}_\varepsilon \oplus (\overline{B}_R \cap [\mathbb{H}_{m,k}^-]^2),
\]
where \( {\bf U}_\varepsilon = (\mathcal{A} u_\varepsilon^m, \mathcal{B} u_\varepsilon^m) \) with $u_\varepsilon^m =\eta_m u_\varepsilon =\phi_\frac{2}{m}(\cdot,0) u_\varepsilon$, $u_\varepsilon$ is given in \eqref{aubin} and \( \mathcal{A,B} > 0 \) such that $\frac{\mathcal{A}}{\mathcal{B}} = \sqrt{\frac{\alpha}{\beta}}$. Let us recall that we are assuming $0\in\Sigma_{\mathcal{N}}$, so $u_\varepsilon^m$ is a truncation centered at a point of the Neumann boundary of the family \eqref{aubin}. Further, we define 
\[
\mathcal{H} = \{ h : \overline{Q}_m^\varepsilon \to \mathbb{H} \text{ continuous}; \ h|_{\partial Q_m^\varepsilon} = \text{id} \},
\]
and let
\begin{equation}
    c_{\varepsilon} = \inf_{h \in \mathcal{H}} \sup_{{\bf U} \in Q_m^\varepsilon} J(h({\bf U})).
\end{equation}
Let $\mu_1 \leq \mu_2$ be the real eigenvalues of the matrix 
$A = \begin{pmatrix} 
a & b \\ 
b & c 
\end{pmatrix} \in M_{2\times2}(\R)$ and ${\bf U}=(u,v)$ with $|{\bf U}|_{\R^2}^2 = u^2 + v^2$. Then, the quadratic form induced by $A$, say
\begin{align}\label{Quadratic-form}
    (A{\bf U},{\bf U})_{\R^2} = au^2 + 2buv + cv^2,
\end{align}
is such that 
\begin{equation}\label{quadratic-remark}
        \mu_1|{\bf U}|_{\R^2}^2 \leq (A{\bf U},{\bf U})_{\R^2} \leq \mu_2|{\bf U}|_{\R^2}^2,
\end{equation}
since $\mu_1$ and $\mu_2$ are the minimum and the maximum respectively of the quadratic form \eqref{Quadratic-form} restricted to the unity sphere.
\begin{lemma}{\rm(\cite[Lemma 6]{Ortega-2024-Subcritical})}\label{lemma6}
    For any $k\in \mathbb{N}$, one has
    \begin{align*}
        \lambda_k^s = \inf_{u \in \mathbb{H}_{k-1}^+} \frac{\|u\|_{\Hspace}^2}{\|u\|_{L^2{(\Omega)}}^2} = \sup_{u \in \mathbb{H}_{k}^-} \frac{\|u\|_{\Hspace}^2}{\|u\|_{L^2{(\Omega)}}^2}.
    \end{align*}
\end{lemma}
\begin{lemma}\label{geometry-1}
   If $\mu_2 < \lambda_{k+1}^s$, there exist constants $\sigma, \rho>0$ such that
   \begin{align*}
       J(u,v) \geq \sigma, \quad \forall (u,v) \in S_\rho:= \partial B_\rho \cap [\mathbb{H}_k^+]^2,
   \end{align*}
   where $B_\rho = \{ (u,v) \in \mathbb{H}: \|(u,v)\|_{\mathbb{H}} < \rho\}$ and $[\mathbb{H}_k^+]^2 = \mathbb{H}_k^+ \times \mathbb{H}_k^+$.
\end{lemma}
\begin{proof}
The energy functional $J:\mathbb{H} \rightarrow \R$ associated with \eqref{Euler Lagrange Identity} is given by
\begin{align*}
    J(u,v)=& \frac{1}{2} \|(u,v)\|_{\mathbb{H}}^2-\frac{1}{2}\int_{\Omega} (A{\bf U},{\bf U})_{\R^2}\,\dx  -\int_{\Omega}|u|^\alpha |v|^\beta \,\dx\\
     \geq& \frac{1}{2} \|(u,v)\|_{\mathbb{H}}^2-\frac{\mu_2}{2}\int_{\Omega} (u^2+v^2)\,\dx -\int_{\Omega}|u|^\alpha |v|^\beta \,\dx\\
     \geq& \frac{1}{2} \left(1 -\frac{\mu_2}{\lambda_{k+1}^s} \right) \|(u,v)\|_{\mathbb{H}}^2-  \frac{\|(u,v)\|_{\mathbb{H}}^{2_s^*} }{ \widetilde{S}_{(\alpha,\beta)}(\Omega)^{\frac{2_s^*}{2}}}\\
      \geq& \frac{1}{2} \left(1 -\frac{\mu_2}{\lambda_{k+1}^s} \right) \|(u,v)\|_{\mathbb{H}}^2-  \frac{\|(u,v)\|_{\mathbb{H}}^{2_s^*} }{ S_{\alpha +\beta}(\Omega)^{\frac{2_s^*}{2}}}\\
    \geq & \frac{1}{2} \left(1 -\frac{\mu_2}{\lambda_{k+1}^s} \right) \rho^2-  \left( \frac{1}{ S_{\alpha +\beta}(\Omega)^{\frac{2_s^*}{2}}}\right) \rho^{2_s^*},
\end{align*}
where the inequalities follow by \eqref{quadratic-remark}, Lemma \ref{lemma6}, \eqref{Sobolev type constant on product space} and Theorem \ref{important relation} (which implies $\widetilde{S}_{(\alpha,\beta)}(\Omega) \geq S_{\alpha +\beta}(\Omega)$). Since $2<2_s^*$, we can choose $\rho>0$ small enough such that $\sigma=\sigma(\rho)$, which is defined as 
\begin{align*}
   \sigma:= \frac{1}{2} \left(1 -\frac{\mu_2}{\lambda_{k+1}^s} \right) \rho^2- \left( \frac{1}{ S_{\alpha +\beta}(\Omega)^{\frac{2_s^*}{2}}}\right) \rho^{2_s^*},
\end{align*}
is positive. Hence, the result follows immediately for such $\rho,\sigma>0$.
\end{proof}
\begin{lemma}\label{geometry-2}
    Let $\lambda_k^s \leq \mu_1 \leq \mu_2 < \lambda_{k+1}^s$. {There exist constants $\varrho_0,R_0$ such that, for $\varrho \geq \varrho_0$ and $R\geq R_0$, we have
\[\max_{(u,v)\in\partial Q_m^\varepsilon} J(u,v) \leq B(m),\quad \text{ where } B(m) \rightarrow 0 \text{ as } m \rightarrow +\infty.\]}
\end{lemma}
\begin{proof}
First we decompose $\partial Q_m^\varepsilon$ as follows: $\partial Q_m^\varepsilon = \Gamma_{1,m} \cup \Gamma_{2,m} \cup \Gamma_{3,m}$, where
\begin{align*}
    \Gamma_{1,m}&= \overline{B}_R \cap [\mathbb{H}_{m,k}^-]^2,\\
    \Gamma_{2,m} &= \left\{(u,v) \in \mathbb{H}: (u,v)= (u_1,v_1) + \varrho {\bf U}_\varepsilon, \text{ with } (u_1,v_1) \in B_R \cap [\mathbb{H}_{m,k}^-]^2 \right\},\\
    \Gamma_{3,m} &= \left\{(u,v) \in \mathbb{H}: (u,v)= (u_1,v_1) + t {\bf U}_\varepsilon, \text{ with } (u_1,v_1) \in [\mathbb{H}_{m,k}^-]^2, \|(u_1,v_1)\|_{\mathbb{H}}=R, 0 \leq t \leq \varrho \right\}.
\end{align*}
\noindent \textbf{Step I: Estimate on $\Gamma_{1,m}$.}  Let $(u,v) \in \Gamma_{1,m}$. Using Lemma \ref{lemma3.2}, \eqref{quadratic-remark}, we obtain the following estimate
\begin{equation}\label{sec4-eq1}
\begin{split}
    J(u,v)&= \frac{1}{2} \|(u,v)\|_{\mathbb{H}}^2-\frac{1}{2}\int_{\Omega} (A{\bf U},{\bf U})_{\R^2}\,\dx - \int_{\Omega}|u|^\alpha |v|^\beta\,\dx \\
    &\leq \frac{1}{2} \left(\lambda_k^s + C m^{2s-N} \right) \int_{\Omega}(u^2+v^2)\,\dx  - \frac{\mu_1}{2} \int_{\Omega}(u^2+v^2)\,\dx - \int_{\Omega}|u|^\alpha |v|^\beta\,\dx \\
    &\leq \frac{1}{2} \left(\lambda_k^s + C m^{2s-N} \right)\int_{\Omega}(u^2+v^2)\,\dx - \frac{\lambda_k^s}{2} \int_{\Omega}(u^2+v^2)\,\dx- \int_{\Omega}|u|^\alpha |v|^\beta\,\dx \\
    &= \frac{1}{2} C m^{2s-N} \int_{\Omega}(u^2+v^2)\,\dx - \int_{\Omega}|u|^\alpha |v|^\beta\,\dx,
\end{split}
\end{equation}
where $C=C(N,s,k)$ depends on $N,s,k$.
Next, we define the function $F_{\theta_m}: \R^2 \rightarrow \R$ given by 
\begin{align*}
    F_{\lambda_m}(u,v) = |u|^\alpha |v|^\beta - \theta_m (u^2 + v^2).
\end{align*}
A pair $(u,v)$ is an extremum point of $F_{\theta_m}$ if
\begin{align}
    \alpha |u|^{\alpha-2}u|v|^\beta - 2 \theta_m u &= 0 \label{4.5}\\
    \beta|u|^\alpha |v|^{\beta-2}v - 2 \theta_m v &=0 \label{4.6}
\end{align}
Now multiplying \eqref{4.5}, \eqref{4.6} with $\beta u$ and $\alpha v$ respectively, and then subtraction yields $|v|= \left(\frac{\beta}{\alpha} \right)^{\frac{1}{2}}|u|$. Further we define
\begin{align}\label{function-f(u)}
    f(u):= F_{\lambda_m}(|u|, \left(\frac{\beta}{\alpha} \right)^{\frac{1}{2}}|u|) = |u|^{\alpha +\beta} \left(\frac{\beta}{\alpha} \right)^{\frac{\beta}{2}} - \theta_m \left(\frac{\alpha +\beta}{\alpha} \right)|u|^2.
\end{align}
Then $f$ has a global minimum at $u_0 = \left(\frac{2 \theta_m}{\alpha} \left(\frac{\alpha}{\beta} \right)^{\frac{\beta}{2}} \right)^{\frac{1}{\alpha+\beta-2}}$, which further implies
\begin{align}\label{inequality-1}
    f(u) \geq f(u_0) 
    &= - C_{\alpha,\beta} \theta_m^{\frac{\alpha+\beta}{\alpha+\beta-2}},
\end{align}
where 
\[C_{\alpha,\beta}:= \left(\frac{\alpha+\beta-2}{\alpha} \right) \left(\frac{\alpha}{\beta} \right)^{\frac{\beta}{\alpha+\beta-2}} \left(\frac{2}{\alpha} \right)^{\frac{2}{\alpha+\beta-2}}.\]
From \eqref{sec4-eq1},\eqref{function-f(u)},\eqref{inequality-1} and setting $\theta_m = \frac{1}{2} C m^{2s-N}$, we deduce that for constant $\tilde{C}>0$,
\begin{align}\label{stepI-estimate}
  J(u,v) &\leq -  \int_{\Omega} \left( |u|^\alpha |v|^\beta - \theta_m (u^2+v^2) \right)\,\dx \leq C_{\alpha,\beta} \theta_m^{\frac{\alpha+\beta}{\alpha+\beta-2}} \leq \tilde{C} m^{- \frac{N}{2s}(N-2s)}.
\end{align}
\noindent \textbf{Step II: Estimate on $\Gamma_{2,m}$.} Let $(u,v) \in \Gamma_{2,m}$. Since \( {\bf U}_\varepsilon = (\mathcal{A} u_\varepsilon^m, \mathcal{B}u_\varepsilon^m) \), where $u_\varepsilon^m =\eta_m u_\varepsilon$, we write
\begin{align*}
    (u,v) = (u_1,v_1) + \varrho (\mathcal{A} u_\varepsilon^m, \mathcal{B} u_\varepsilon^m).
\end{align*}
Notice that $\text{supp}(u_1) \cap \text{supp}(u_\varepsilon^m)=\emptyset$ and $\text{supp}(v_1) \cap \text{supp}(u_\varepsilon^m)=\emptyset$. We can write $J(u,v)$ on $\Gamma_{2,m}$ as follows
\begin{equation}\label{functional-g1}
\begin{split}
    J(u,v) =& J(u_1,v_1) + { \left( J(\varrho \mathcal{A} u_\varepsilon^m,\varrho \mathcal{B} u_\varepsilon^m)) +2\varrho \mathcal{A} \langle u_1, u_\epsilon^m \rangle_{\Hspace} + 2\varrho \mathcal{B} \langle v_1, u_\epsilon^m \rangle_{\Hspace}\right)}\\
    :=&I_{m,1}+ {I_{m,2}}.
    \end{split}
\end{equation}
Following Step I, we get
\begin{align}\label{estimate-Im1}
  I_{m,1} \leq C m^{- \frac{N}{2s}(N-2s)},
\end{align}
for a constant $C>0$. On the other hand, since
\begin{align*}
    I_{m,2} \leq& \frac{1}{2} \|(\varrho \mathcal{A} u_\varepsilon^m,\varrho \mathcal{B} u_\varepsilon^m)\|_{\mathbb{H}}^2 - \varrho^{2_s^*} \mathcal{A}^\alpha \mathcal{B}^\beta \int_{\Omega} |u_\varepsilon^m|^{2_s^*}\,\dx+ { 2 \varrho \left( \mathcal{A} \|u_1\|_{\Hspace} + \mathcal{B} \|v_1\|_{\Hspace}\right) \|  u_\varepsilon^m\|_{\Hspace} }\\
    =& \frac{\varrho^2}{2} (\mathcal{A}^2 + \mathcal{B}^2) \| u_\varepsilon^m\|_{\Hspace}^2 - \varrho^{2_s^*} \mathcal{A}^\alpha \mathcal{B}^\beta \int_{\Omega} |u_\varepsilon^m|^{2_s^*}\,\dx + { 2 \varrho \left( \mathcal{A} \|u_1\|_{\Hspace} + \mathcal{B} \|v_1\|_{\Hspace}\right) \|  u_\varepsilon^m\|_{\Hspace} }.
\end{align*}
{and $2<2_s^*$, there exist $\varrho_0>0$ large enough such that $I_{m,2} \leq 0$ for $\varrho \geq \varrho_0$.}
{Therefore, from \eqref{estimate-Im1}, we conclude $J(u,v) \leq B(m)$ on $\Gamma_{2,m}$ for $\varrho \geq \varrho_0$ with $B(m)$ such that $B(m) \rightarrow 0$ as $m \rightarrow +\infty$.}

\noindent \textbf{Step III: Estimate on $\Gamma_{3,m}$.}
For $(u,v) \in \Gamma_{3,m}$ and \( {\bf U}_\varepsilon = (\mathcal{A} u_\varepsilon^m, \mathcal{B}u_\varepsilon^m) \), we can write 
\[(u,v) = (u_1,v_1) + t (\mathcal{A} u_\varepsilon^m, \mathcal{B} u_\varepsilon^m),\]
where $t\in [0,\varrho]$ with $\varrho \geq \varrho_0$. Also, we have $(u_1,v_1) \in [\mathbb{H}_{m,k}^-]^2$ and $\|(u_1,v_1)\|_{\mathbb{H}}=R$. Notice that
\[J(u,v) = J(u_1,v_1) + J(t \mathcal{A} u_\varepsilon^m,t \mathcal{B}u_\varepsilon^m) + I_{m,4},\]
where
\[I_{m,4}= {2t \mathcal{A} \langle u_1, u_\epsilon^m \rangle_{\Hspace} + 2t \mathcal{B} \langle v_1, u_\epsilon^m \rangle_{\Hspace}}.
\]
We divide this step into two parts:

\noindent \textbf{(a) If \boldsymbol{$\lambda_k^s < \mu_1 \leq \mu_2 < \lambda_{k+1}^s$}:} In this case, using \eqref{quadratic-remark} and Lemma \ref{lemma3.2}, we have
\begin{equation}\label{5.13}
\begin{split}
    J(u_1,v_1)&\leq \frac{1}{2} \|(u_1,v_1)\|_{\mathbb{H}}^2-\frac{1}{2}\int_{\Omega} (A(u_1,v_1),(u_1,v_1))_{\R^2}\,\dx \\
    &\leq \frac{1}{2} \|(u_1,v_1)\|_{\mathbb{H}}^2 - \frac{\mu_1}{2} \left( \|u_1\|_{L^2(\Omega)}^2 + \|v_1\|_{L^2(\Omega)}^2 \right)  \\
    & \leq \frac{1}{2} \|(u_1,v_1)\|_{\mathbb{H}}^2 - \frac{\mu_1}{2} \left( \frac{\|(u_1,v_1)\|_{\mathbb{H}}^2}{\lambda_k^s + C(N,s,k)m^{2s-N}} \right) \\
    & \leq  \frac{(\lambda_k^s-\mu_1) \|(u_1,v_1)\|_{\mathbb{H}}^2}{2(\lambda_k^s + C(N,s,k)m^{2s-N})}  + \frac{C(N,s,k)m^{2s-N}}{2(\lambda_k^s + C(N,s,k)m^{2s-N})} \|(u_1,v_1)\|_{\mathbb{H}}^2\\
    &\leq  \frac{(\lambda_k^s-\mu_1)}{2(\lambda_k^s + C(N,s,k))} R^2  + \frac{C(N,s,k)m^{2s-N}}{2 \lambda_k^s } R^2.
\end{split}
\end{equation}
{Next,}
\begin{equation} \label{5.14}
\begin{split}
    J(t \mathcal{A}u_\varepsilon^m,t \mathcal{B}u_\varepsilon^m) \leq& \frac{1}{2} \|(t \mathcal{A}u_\varepsilon^m,t \mathcal{B}u_\varepsilon^m)\|_{\mathbb{H}}^2 - t^{2_s^*} \mathcal{A}^\alpha \mathcal{B}^\beta \int_{\Omega} |u_\varepsilon^m|^{2_s^*}\,\dx \\
    =& \frac{t^2}{2} (\mathcal{A}^2 + \mathcal{B}^2) \| u_\varepsilon^m\|_{\Hspace}^2 - t^{2_s^*} \mathcal{A}^\alpha \mathcal{B}^\beta \int_{\Omega} |u_\varepsilon^m|^{2_s^*}\,\dx. \\
    %\leq& \frac{t^2}{2} (\mathcal{A}^2 + \mathcal{B}^2) \left({\frac{1}{2}}\| U_\varepsilon\|_{\Xspace}^2 + O \left( (\varepsilon m)^{N-2s}\right) \right)\\
    %& - t^{2_s^*} \mathcal{A}^\alpha \mathcal{B}^\beta \left( {\frac{1}{2}}\|u_\varepsilon\|_{L^{2_s^*}(\R^N)}^{2_s^*} + O((\varepsilon m)^N) \right).
\end{split}
\end{equation}
{Thus, from \eqref{5.13} and \eqref{5.14},} we obtain
\begin{align*}
    J(u,v) \leq& \frac{(\lambda_k^s-\mu_1)}{2(\lambda_k^s + C(N,s,k))} R^2   + \frac{C(N,s,k)m^{2s-N}}{2 \lambda_k^s } R^2 + \frac{t^2}{2} (\mathcal{A}^2 + \mathcal{B}^2) {\| u_\varepsilon^m\|_{\Hspace}^2} - t^{2_s^*} \mathcal{A}^\alpha \mathcal{B}^\beta {\int_{\Omega} |u_\varepsilon^m|^{2_s^*}\,\dx}\\
    & + { 2 t \left( \mathcal{A} \|u_1\|_{\Hspace} + \mathcal{B} \|v_1\|_{\Hspace}\right) \| u_\varepsilon^m\|_{\Hspace} }.
\end{align*}
Since $\lambda_k^s -\mu_1<0$, we can choose $R_1$ large enough such that for every $R \geq R_1$, we have
\begin{align*}
    \frac{(\lambda_k^s-\mu_1)}{2(\lambda_k^s + C(N,s,k))} R^2 + \frac{t^2}{2} (\mathcal{A}^2 + \mathcal{B}^2) r_{\varepsilon,m}^{(1)}  - t^{2_s^*} \mathcal{A}^\alpha \mathcal{B}^\beta r_{\varepsilon,m}^{(2)} + {2 t \left( \mathcal{A} \|u_1\|_{\Hspace} + \mathcal{B} \|v_1\|_{\Hspace}\right) r_{\varepsilon,m}^{(3)}} <0,
\end{align*}
where {$r_{\varepsilon,m}^{(1)}:=\| u_\varepsilon^m\|_{\Hspace}^2$, $r_{\varepsilon,m}^{(2)}=\|u_\varepsilon^m\|_{L^{2_s^*}(\Omega)}^{2_s^*}$ and $r_{\varepsilon,m}^{(3)}=\| u_\varepsilon^m\|_{\Hspace}$}. Hence, we conclude that {$J(u,v) \leq C(m)$ on $\Gamma_{3,m}$ for any $R\geq R_1$ with $C(m)$ such that $C(m) \rightarrow 0$ as $m \rightarrow +\infty$.}

\noindent \textbf{(b) If \boldsymbol{$\lambda_k^s = \mu_1 \leq \mu_2 < \lambda_{k+1}^s$}:} Let $(u_1,v_1) = R (\Bar{u}_1, \Bar{v}_1)$ with $\|(\Bar{u}_1, \Bar{v}_1)\|_{\mathbb{H}}=1$. For $q,d,c_1,c_2 \in R$, we set
\begin{align*}
    (\Bar{u}_1, \Bar{v}_1) = c_1(z_1,z_2) + c_2 (q\varphi_k^m, d\varphi_k^m),
\end{align*}
where $(q\varphi_k^m, d\varphi_k^m) \in \mathcal{N}:=\text{span}\left\{(\varphi_k^m, 0), (0, \varphi_k^m) \right\}$ and $(z_1,z_2) \in [\mathbb{H}_{m,k}^-]^2 \cap \mathcal{N}^\perp$. Notice that
%%%%%%%%%%%%%%%%%%%%%%%%%% 1 %%%%%%%%%%%%%%%%%%%%%%%%%%%%%%
\begin{equation}\label{5.15}
\begin{split}
    J (u_1, v_1) &\leq \frac{R^2}{2} \| (\Bar{u}_1, \Bar{v}_1) \|_{\mathbb{H}}^2 - \frac{1}{2} \int_{\Omega} ( A(R \Bar{u}_1, R \Bar{v}_1) , (R \Bar{u}_1, R \Bar{v}_1) )_{\R^2} \, \dx\\
    &\leq \frac{R^2}{2} \| (\Bar{u}_1, \Bar{v}_1) \|_{\mathbb{H}}^2 - R^2 \frac{\mu_1}{2} \int_{\Omega} \left( \Bar{u}_1^2 + \Bar{v}_1^2 \right) \,\dx.
    \end{split}
\end{equation}
Now observe that
\begin{equation}\label{n1}
\begin{split}
    \| (\Bar{u}_1, \Bar{v}_1) \|_{\mathbb{H}}^2 =& \| \Bar{u}_1 \|_{\Hspace}^2 + \| \Bar{v}_1 \|_{\Hspace}^2 \\
    =& \| c_1 z_1 + c_2 q\varphi_k^m \|_{\Hspace}^2 + \| c_1 z_2 + c_2 d \varphi_k^m \|_{\Hspace}^2 \\
    =& \| c_1 z_1 \|_{\Hspace}^2 + \| c_2 q\varphi_k^m \|_{\Hspace}^2 + \| c_1 z_2 \|_{\Hspace}^2 + \| c_2 d \varphi_k^m \|_{\Hspace}^2 \\
    &+ 2 \langle c_1 z_1, c_2q \varphi_k^m \rangle_{\Hspace} + 2 \langle c_1 z_2, c_2 d \varphi_k^m \rangle_{\Hspace} \\
    =& \| (c_1 z_1, c_1 z_2) \|_{\mathbb{H}}^2 + \| (c_2 q\varphi_k^m, c_2 d \varphi_k^m) \|_{\mathbb{H}}^2 + 2 c_1 c_2q \langle z_1, \varphi_k^m \rangle_{\Hspace} + 2 c_1 c_2 d \langle z_2, \varphi_k^m \rangle_{\Hspace}
\end{split}
\end{equation}
From \eqref{5.15} and \eqref{n1}, we obtain that
\begin{align*}
  J (u_1, v_1) \leq& \frac{R^2}{2} \left( \| (c_1 z_1, c_1 z_2) \|_{\mathbb{H}}^2 + \| (c_2 q\varphi_k^m, c_2 d \varphi_k^m) \|_{\mathbb{H}}^2  + 2 c_1 c_2q \langle z_1, \varphi_k^m \rangle_{\Hspace} + 2 c_1 c_2 d \langle z_2, \varphi_k^m \rangle_{\Hspace} \right)\\
  & - \frac{R^2 \mu_1}{2}  \left( \| c_1 z_1 \|_{L^2(\Omega)}^2 + \| c_1 z_2 \|_{L^2(\Omega)}^2 + \| c_2q \varphi_k^m \|_{L^2(\Omega)}^2 + \| c_2 d \varphi_k^m \|_{L^2(\Omega)}^2 \right) \\
  & - R^2 \mu_1 \left(  c_1 c_2q \int_{\Omega} z_1 \varphi_k^m \, \dx +  c_1 c_2 d \int_{\Omega} z_2 \varphi_k^m \, \dx \right).
\end{align*}
Since \(\mu_1 = \lambda_k^s\), we write
\begin{align*}
    J (u_1, v_1) \leq& \frac{R^2}{2} \left( c_1^2 \| (z_1, z_2) \|_{\mathbb{H}}^2 + c_2^2 \| (q\varphi_k^m, d \varphi_k^m) \|_{\mathbb{H}}^2 \right)\\
    & - \frac{R^{2}}{2} \lambda_k^s \left( c_1^2 \left( \| z_1 \|_{L^2(\Omega)}^2 + \| z_2 \|_{L^2(\Omega)}^2 \right) + c_2^2 \left( \| q\varphi_k^m \|_{L^2(\Omega)}^2 + \| d \varphi_k^m \|_{L^2(\Omega)}^2  \right)\right) + R^2 \mathcal{T}_m\\
    \leq&  \frac{R^2}{2} c_1^2 \left( \lambda_{k-1}^s + C_{k-1} m^{2s-N} \right) \left( \| z_1 \|_{L^2(\Omega)}^2 + \| z_2 \|_{L^2(\Omega)}^2 \right) + \frac{R^2}{2} c_2^2 \left( \lambda_k^s + C_k m^{2s-N} \right) (q^2 + d^2) \| \varphi_k^m \|_{L^2(\Omega)}^2\\
    & - \frac{R^{2}}{2} \lambda_k^s \left( c_1^2 \big( \| z_1 \|_{L^2(\Omega)}^2 + \| z_2 \|_{L^2(\Omega)}^2 \big) + c_2^2 (q^2 + d^2) \| \varphi_k^m \|_{L^2(\Omega)}^2 \right) + R^2 \mathcal{T}_m,
\end{align*}
where 
\begin{align*}
    \mathcal{T}_m =   c_1 c_2q \langle z_1, \varphi_k^m \rangle_{\Hspace} + d c_1 c_2 \langle z_2, \varphi_k^m \rangle_{\Hspace} - \mu_1 qc_1 c_2 \int_{\Omega} z_1 \varphi_k^m \, \dx - \mu_1 d c_1 c_2  \int_{\Omega} z_2 \varphi_k^m \, \dx.
\end{align*}
Notice that \( \mathcal{T}_m \to 0 \) as \( m \to +\infty \) using Lemma \ref{orthogonality of approximate eigenfunctions} since $(z_1,z_2) \in [\mathbb{H}_{m,k}^-]^2 \cap \mathcal{N}^\perp$.  Thus, we obtain
\begin{align*}
    J (u_1, v_1) \leq \frac{R^2}{2} (\lambda_{k-1}^s - \lambda_k^s) c_1^2 \left( \| z_1 \|_{L^2(\Omega)}^2 + \| z_2 \|_{L^2(\Omega)}^2 \right) + R^2 \widetilde{C}_k m^{2s-N} + R^2 \mathcal{T}_m.
\end{align*}
Finally, we obtain
\begin{align*}
    J (u, v) \leq& J (u_1, v_1) + J (t \mathcal{A} u_\varepsilon^m, t \mathcal{B} u_\varepsilon^m) + I_{m,4}\\
     \leq& \frac{R^2}{2} (\lambda_{k-1}^s - \lambda_k^s) c_1^2 \left( \| z_1 \|_{L^2(\Omega)}^2 + \| z_2 \|_{L^2(\Omega)}^2 \right) + R^2 \widetilde{C}_k m^{2s-N} + R^2 \mathcal{T}_m\\
    & + \frac{t^2}{2} (\mathcal{A}^2 + \mathcal{B}^2) {\| u_\varepsilon^m\|_{\Hspace}^2} - \frac{t^{2_s^*}}{2} \mathcal{A}^\alpha \mathcal{B}^\beta {\|  u_\varepsilon^m \|_{L^{2_s^*}(\Omega)}^{2_s^*}+ I_{m,4}}.
\end{align*}
Since $\lambda_{k-1}^s - \lambda_k^s<0$,  if \( c_1 \neq 0 \), then there exists \( R_2 \) large enough such that for all $R> R_2$ we can get
\begin{align*}
    \frac{R^2}{2} (\lambda_{k-1}^s - \lambda_k^s) c_1^2 \left( \| z_1 \|_{L^2(\Omega)}^2 + \| z_2 \|_{L^2(\Omega)}^2 \right) + \frac{t^2}{2} (\mathcal{A}^2 + \mathcal{B}^2) {\| u_\varepsilon^m\|_{\Hspace}^2} + {I_{m,4}} <0.
\end{align*}
Hence, it follows that {$J (u, v) \leq D(m)$ for any $R\geq R_2$ with $D(m) \to 0$ as $m \to +\infty$}.

If \( c_1 = 0 \), then \( (\Bar{u}_1, \Bar{v}_1) = (c_2 q\varphi_k^m, c_2 d \varphi_k^m) \) and
\begin{equation}\label{5.17}
\begin{split}
    J (u_1, v_1) \leq& \frac{R^2}{2} \| (c_2q \varphi_k^m, c_2 d \varphi_k^m) \|_{\mathbb{H}}^2 - \frac{R^{2} \lambda_k^s}{2} \left(  \| c_2 q\varphi_k^m\|_{L^2(\Omega)}^2 + \|c_2 d \varphi_k^m \|_{L^2(\Omega)}^2 \right)\\ 
    &- R^{2_s^*}c_2^{2_s^*}q^\alpha d^\beta \int_{\Omega} |\varphi_k^m|^{2_s^*}\,\dx \\
    \leq& \frac{R^2 c_2^2}{2} (q^2 + d^2) \left( \lambda_k^s + C_k m^{2s-N} \right) \| \varphi_k^m \|_{L^2(\Omega)}^2 - \frac{R^2 c_2^2}{2}(q^2 + d^2) \lambda_k^s \| \varphi_k^m \|_{L^2(\Omega)}^2 \\
    & - R^{2_s^*}c_2^{2_s^*}q^\alpha d^\beta \int_{\Omega} |\varphi_k^m|^{2_s^*}\,\dx.
\end{split}
\end{equation}
Now we claim that there exists \( \Theta_k > 0 \) such that
\begin{align}\label{claim}
    \int_{\Omega} |\varphi_k^m|^{2_s^*}\, \dx \geq \Theta_k > 0.
\end{align}
By Hölder inequality we have $\| \varphi_k^m \|_{L^2(\Omega)} \leq C_1 \| \varphi_k^m \|_{L^{2_s^*} (\Omega)}$, where $C_1 = |\Omega|^{\frac{1}{2} - \frac{1}{2_s^*}}$. Using \eqref{test function em}, we can write
\[
\int_{\Omega} \left(\varphi_k^m \right)^2\, \dx \geq \int_{\Omega \setminus B_{2/m}} \varphi_k^2\, \dx.
\]
Since $\|\varphi_k\|_{L^2(\Omega)}=1$, we can write
\begin{align*}
    \int_{\Omega \setminus B_{2/m} } \left(\varphi_k\right)^2\, \dx = 1 - \int_{\Omega \cap B_{2/m} } \left(\varphi_k \right)^2\, \dx.
\end{align*}
We can find a $m=m_1>0$ such that $\int_{\Omega \cap B_{2/m_1} } \left(\varphi_k \right)^2\, \dx \leq c_4$, where $c_4 \in (0,1)$. This yields
\begin{align*}
    \int_{\Omega \setminus B_{2/m_1} } \left(\varphi_k\right)^2\, \dx \geq  1 - c_4,
\end{align*}
which further implies $\int_{\Omega} \left(\varphi_k^m\right)^2 dx \geq (1 - c_4)$. Hence, we obtain that $\| \varphi_k^m \|_{L^{2_s^*} (\Omega)} \geq \frac{\sqrt{1 - c_4}}{C_1} > 0$. We can choose $\Theta_k = \left(\frac{\sqrt{1 - c_4}}{C_1} \right)^{2_s^*}$, which proves the claim \eqref{claim}. From \eqref{5.17}, there exists a number \( R_3 > 0 \) such that for all \( R > R_3 \), we have
\begin{align*}
    J (u, v) =& J (u_1, v_1) + J ({t \mathcal{A}u_\varepsilon^m, t \mathcal{B} u_\varepsilon^m}) + I_{m,4} \\
    \leq& R^2 \widetilde{C}_k m^{2s-N} \| \varphi_k^m \|_{L^2(\Omega)}^2 - C R^{2_s^*} \Theta_k + \frac{t^2}{2} (\mathcal{A}^2 + \mathcal{B}^2) { \| u_\varepsilon^m \|_{\Hspace}^2}- \frac{t^{2_s^*}}{2} \mathcal{A}^\alpha \mathcal{B}^\beta {\| u_{\varepsilon}^m \|_{L^{2_s^*}(\Omega}^{2_s^*}} + I_{m,4}.
\end{align*}
which completes the proof. To conclude the Step III, we choose $R_0> \max\{R_1,R_2,R_3\}$.
\end{proof}

\section{Main result}\label{Main_Result}
\begin{lemma}\label{energy-threshold-estimate}
    Let {$N > 8s$} and $\lambda_k^s \leq \mu_1 \leq \mu_2<\lambda_{k+1}^s$. Then
    \begin{align}\label{Threshold}
    \max\limits_{(u,v) \in Q_m^\varepsilon} J(u,v) < \frac{s}{N} (2_s^*)^{-\frac{N-2s}{2s}}\left(\widetilde{S}_{(\alpha,\beta)}(\Omega) \right)^{\frac{N}{2s}}.
\end{align}
\end{lemma}
\begin{proof}
    For a given $(u,v) \in Q_m^\varepsilon$ and \( {\bf U}_\varepsilon = (\mathcal{A} u_\varepsilon^m, \mathcal{B}u_\varepsilon^m) \), we write 
    \[(u,v) = (u_1,v_1) + t (\mathcal{A} u_\varepsilon^m, \mathcal{B} u_\varepsilon^m),\] where $(u_1,v_1) \in  \overline{B}_R \cap [\mathbb{H}_{m,k}^-]^2$ and {$t\in [0, \varrho]$}. We recall that
\begin{align}\label{5.1}
    J(u,v) = J(u_1,v_1) + J(t \mathcal{A} u_\varepsilon^m,t \mathcal{B}u_\varepsilon^m) + I_{m,5},
\end{align}
where
\begin{equation}\label{Im5}
    I_{m,5} = {2t \mathcal{A} \langle u_1, u_\epsilon^m \rangle_{\Hspace} + 2t \mathcal{B} \langle v_1, u_\epsilon^m \rangle_{\Hspace}}.
\end{equation}
{In the following the same constant $C_i, i\in \mathbb{N},$ might be different at each step.
Let us evaluate the term
\begin{align}\label{IP-1}
    \langle u_1, u_\epsilon^m \rangle_{\Hspace} =\sum_{i=1}^{k}\alpha_i \langle \varphi_i^m, u_\epsilon^m \rangle_{\Hspace} 
\end{align}
Now notice that
\begin{align}\label{Inq-3}
    \langle \varphi_i^m, u_\epsilon^m \rangle_{\Hspace} &= \langle \varphi_i^m-\varphi_i, u_\epsilon^m \rangle_{\Hspace} + \langle \varphi_i, u_\epsilon^m \rangle_{\Hspace} \notag \\
    &=\langle \varphi_i^m-\varphi_i, u_\epsilon^m \rangle_{\Hspace} + \sum_{j=1}^{\infty} \lambda_j^s \langle \varphi_i, \varphi_j \rangle_2 \langle  u_\epsilon^m, \varphi_j \rangle_2 \notag\\
    &\leq C_1 m^{\frac{2s-N}{2}} \left(\frac{1}{2}\|U_\varepsilon\|_{\Xspace} + C_2 \varepsilon^\frac{N-2s}{2}m^\frac{N-2s}{2}\right) + \lambda_i^s  \langle  u_\epsilon^m, \varphi_i\rangle_2  \notag\\
    &\leq C_1 m^{\frac{2s-N}{2}}\|U_\varepsilon\|_{\Xspace} + C_2 \varepsilon^\frac{N-2s}{2}+ C_3 \varepsilon^\frac{N-2s}{2} m^{-2s}.
\end{align}
The last inequality follows using the following estimate
\begin{align*}
    \left|\langle u_\varepsilon^m,\varphi_j \rangle_2 \right| \leq \int_{\Omega} | \eta_m(x) u_\varepsilon(x) \varphi_i(x)|\,\dx \leq \|\varphi_i\|_\infty \int_{\Omega \cap B_{\frac{1}{m}}} | u_\varepsilon(x)|\,\dx \leq C \varepsilon^{\frac{N-2s}{2}} m^{-2s}.
\end{align*}
Since $t\in[0,\varrho]$ for some $\varrho>0$, from \eqref{IP-1} and \eqref{Inq-3}, we deduce that
\begin{align}\label{est-5.7}
     I_{m,5} \leq \varrho (\mathcal{A} + \mathcal{B}) \left( C_4 m^{\frac{2s-N}{2}}\|U_\varepsilon\|_{\Xspace} + C_5 \varepsilon^\frac{N-2s}{2}+ C_6 \varepsilon^\frac{N-2s}{2} m^{-2s} \right).
\end{align}
}
From \eqref{stepI-estimate} in Step I, we get
\begin{align}\label{5.5} 
    J(u_1,v_1) \leq \tilde{C} m^{- \frac{N}{2s}(N-2s)}.
\end{align}
Next, {since $\frac{1}{2m}> \left(\frac{1}{m} \right)^{(l+1)}$for $l=\frac{1}{2}$ and $m>4$}, we have the following $L^2$-estimate,
\begin{equation*}
\begin{split}
    \|u_\varepsilon^m\|_2^2 &= \int_{\Omega}|\eta_m u_\varepsilon|^2\,\dx\geq C\int_{|x|< \left(\frac{1}{m} \right)^{(l+1)}} \frac{\varepsilon^{N-2s}}{\left(\varepsilon^2 + |x|^2 \right)^{N-2s}}\,\dx\geq C\frac{\varepsilon^{N-2s}}{\left(\varepsilon^2 + \left(\frac{1}{m} \right)^{2(l+1)} \right)^{N-2s}} \int_{|x|< \left(\frac{1}{m} \right)^{(l+1)}} \dx\\
    %& = C \frac{\varepsilon^{N-2s}}{\left(\varepsilon^2 + \left(\frac{1}{m} \right)^{2(l+1)} \right)^{N-2s}} \int_{0}^{\left(\frac{1}{m} \right)^{(l+1)}} r^{N-1}\,\mathrm{d}r\\
    &= C \frac{\varepsilon^{N-2s}}{\left(\varepsilon^2 + \left(\frac{1}{m} \right)^{2(l+1)} \right)^{N-2s}}  \left(\frac{1}{m} \right)^{N(l+1)},
\end{split}
\end{equation*}
for some constant $C>0$. {Thus, letting $\varepsilon = \left(\frac{1}{m} \right)^{l+1}$ with $l=\frac{1}{2}$, we get
\begin{equation}\label{est1}
    \|u_\varepsilon^m\|_2^2 \geq C \frac{m^{-(l+1)(N-2s)}}{\left(\left(\frac{1}{m} \right)^{2(l+1)} + \left(\frac{1}{m} \right)^{2(l+1)} \right)^{N-2s}}  \left(\frac{1}{m} \right)^{N(l+1)} \geq \overline{C} m^{-3s}.
\end{equation}}
{To continue, let us define $g(t):=J(t \mathcal{A} u_\varepsilon^m,t \mathcal{B}u_\varepsilon^m)$. Then, {taking $\varepsilon = \left(\frac{1}{m} \right)^{l+1}$ with $l=\frac{1}{2}$, we get }
\begin{align*}
   g(t)%\leq& \frac{t^2}{2}\left(\mathcal{A}^2 + \mathcal{B}^2 \right)\|u_\varepsilon^m\|_{\Hspace}^2 - \frac{\mu_1 t^2}{2}\left(\mathcal{A}^2 + \mathcal{B}^2 \right) \int_{\Omega} |u_\varepsilon^m|^2\,\dx - t^{\alpha+\beta} \mathcal{A}^\alpha \mathcal{B}^{\beta} \int_{\Omega} |u_\varepsilon^m|^{\alpha +\beta}\,\dx \\
     \leq& \frac{t^2}{2}\left(\mathcal{A}^2 + \mathcal{B}^2 \right) \left(\|u_\varepsilon^m\|_{\Hspace}^2 -\mu_1 \int_{\Omega} |u_\varepsilon^m|^2\,\dx  \right) - t^{\alpha+\beta} \mathcal{A}^\alpha \mathcal{B}^{\beta} \int_{\Omega} |u_\varepsilon^m|^{\alpha +\beta}\,\dx\\
    \leq& \frac{t^2}{2}\left(\mathcal{A}^2 + \mathcal{B}^2 \right) \left(\frac{1}{2}\|U_\varepsilon\|_{\Xspace}^2 + C_2 \varepsilon^{N-2s}m^{N-2s}-C_3 \varepsilon^{N-2s}m^{(l+1)(N-4s)}  \right)\\
    & - t^{\alpha+\beta} \mathcal{A}^\alpha \mathcal{B}^{\beta} \left( \frac{1}{2}\|u_\varepsilon\|_{L^{2_s^*}(\R^N)}^{2_s^*}- C_4 \varepsilon^N m^N \right)\\
   =&{ \frac{t^2}{2}\left(\mathcal{A}^2 + \mathcal{B}^2 \right) \left(\frac{1}{2}\|U_\varepsilon\|_{\Xspace}^2 + C_2 m^{-\frac{N-2s}{2}}-C_3 m^{-3s}  \right) - t^{\alpha+\beta} \mathcal{A}^\alpha \mathcal{B}^{\beta} \left( \frac{1}{2}\|u_\varepsilon\|_{L^{2_s^*}(\R^N)}^{2_s^*}- C_4 m^{-\frac{N}{2}} \right)}.
\end{align*}
Since by the hypothesis \eqref{non_att_hyp}, we have $S_{\alpha+\beta}(\Omega)=2^{-\frac{2s}{N}}S(N,s)$, using Theorem \ref{important relation} we obtain 
\begin{equation}\label{estimate-on-g}
\begin{split}
    \max\limits_{t>0} g(t) \leq&\frac{1}{2}\frac{s}{N}\frac{1}{(2_s^*)^{\frac{2}{2_s^*-2}}}\left[\left(\frac{\mathcal{A}}{\mathcal{B}}\right)^{\frac{2\beta}{\alpha+\beta}}+\left(\frac{\mathcal{B}}{\mathcal{A}}\right)^{\frac{2\alpha}{\alpha+\beta}}\right]^{\frac{2_s^*}{2_s^*-2}}[S(N,s)]^{\frac{2_s^*}{2_s^*-2}}+ C_2 m^{-\frac{N-2s}{2}} - C_3 m^{-3s}  + C_4 m^{-\frac{N}{2}}\\ 
    =&\frac{s}{N} \left(\frac{\widetilde{S}_{(\alpha,\beta)}(\Omega)}{(2_s^*)^{\frac{N-2s}{N}}} \right)^{\frac{N}{2s}} + C_2 m^{-\frac{N-2s}{2}} - C_3 m^{-3s} + C_4 m^{-\frac{N}{2}}.
\end{split}
\end{equation}}
{Then, by \eqref{est-5.7}, \eqref{5.5} and \eqref{estimate-on-g}, we get
\begin{align*}
\max\limits_{(u,v) \in Q_m^\varepsilon} J(u,v)
&\leq  \tilde{C} m^{- \frac{N}{2s}(N-2s)} + \frac{s}{N} \left(\frac{\widetilde{S}_{(\alpha,\beta)}(\Omega)}{(2_s^*)^{\frac{N-2s}{N}}} \right)^{\frac{N}{2s}} + C_2 m^{-\frac{N-2s}{2}} - C_3 m^{-3s} + C_4 m^{-\frac{N}{2}}\\
&\quad + \varrho (\mathcal{A} + \mathcal{B}) \left( C_5 m^{\frac{2s-N}{2}}\|U_\varepsilon\|_{\Xspace} + C_6 m^{-\frac{3(N-2s)}{4}}+ C_7 m^{-\frac{3(N-2s)}{4}-2s} \right).
\end{align*}
To conclude, we observe that, for $N>8s$, we have
\begin{align*}
    -3s> \max\left\{ -\frac{N}{2s}(N-2s) , -\frac{N-2s}{2}, -\frac{N}{2}, -\frac{3(N-2s)}{4}\right\}.
\end{align*}}
Hence, for $m$ large enough we get
\begin{align*}
    \max\limits_{(u,v) \in Q_m^\varepsilon} J(u,v) < \frac{s}{N} (2_s^*)^{-\frac{N-2s}{2s}}\left(\widetilde{S}_{(\alpha,\beta)}(\Omega) \right)^{\frac{N}{2s}}.
\end{align*}
\end{proof}
We recall that the sequence $\{(u_n,v_n)\} \subset E$ is called a (PS) sequence for \(J\) at level \(c\) if
\[
J(u_n,v_n) \to c \quad \text{and} \quad J'(u_n,v_n) \to 0 \text{ in } E' \quad\text{as}\ n \to +\infty.
\]
\begin{lemma}\label{boundedness}
Let $\{(u_n,v_n)\} \subset \mathbb{H}$ be a (PS) sequence for \(J\); then there exists \((u,v)\) such that, up to a subsequence, \((u_n,v_n) \rightharpoonup (u,v)\) in $\mathbb{H}$. Moreover, if \(J(u_n,v_n) \to c\) with \(0 < c < c_0\) and 
\begin{equation}\label{crit_level}
c_0:= \frac{s}{N} (2_s^*)^{-\frac{N-2s}{2s}}\left(\widetilde{S}_{(\alpha,\beta)}(\Omega) \right)^{\frac{N}{2s}},
\end{equation} 
then \((u,v) \neq (0,0)\).
\end{lemma}
\begin{proof}
Let $\{(u_n,v_n)\}$ be a \((PS)_c\) sequence, i.e.,
\begin{equation}\label{5.9}
\frac{1}{2} \|(u_n,v_n)\|^2_{\mathbb{H}} - \frac{1}{2} \int_{\Omega} (au_n^2 + 2b u_nv_n + cv_n^2)\, \dx - \int_{\Omega} |u_n|^{\alpha} |v_n|^{\beta} \, \dx = c + o(1)
\end{equation}
and
\begin{equation}\label{5.10}
\|(u_n,v_n)\|^2_{\mathbb{H}} - \int_{\Omega} (au_n^2 + 2b u_nv_n + cv_n^2)\, \dx - 2_s^* \int_{\Omega} |u_n|^{\alpha} |v_n|^{\beta} \, \dx = o(1). 
\end{equation}
From \eqref{5.9} and \eqref{5.10}, we obtain
\begin{equation}\label{5.11}
2J(u_n,v_n) - \langle J'(u_n,v_n), (u_n,v_n) \rangle = (2_s^* - 2) \int_{\Omega} |u_n|^{\alpha} |v_n|^{\beta} \, \dx \leq 2c + o(1).
\end{equation}
From \eqref{5.9}, \eqref{5.11} and \eqref{inequality-1}, we yield that
\begin{align*}
\|(u_n,v_n)\|^2_{\mathbb{H}} &= 2J(u_n,v_n) + \int_{\Omega} (au_n^2 + 2b u_nv_n + cv_n^2)\, \dx + 2 \int_{\Omega} |u_n|^{\alpha} |v_n|^{\beta} \, \dx \\
&\leq 2J(u_n,v_n) + \mu_2 \int_{\Omega} (u_n^2 + v_n^2) \, \dx + 2 \int_{\Omega} |u_n|^{\alpha} |v_n|^{\beta} \, \dx \\
&\leq C,
\end{align*}
and therefore we conclude that the sequence $(u_n,v_n)$ is bounded in $\mathbb{H}$. Then, there exists a subsequence again denoted by $\{(u_n,v_n)\}$ such that \((u_n,v_n) \rightharpoonup (u,v)\) weakly in $\mathbb{H}$. We claim that \((u,v) \neq (0,0)\). Assume by contradiction that \((u,v) = (0,0)\). We know that
\[
\langle J'(u_n,v_n), (u_n,v_n) \rangle = o(1).
\]
Since the embedding \(\Hspace \hookrightarrow L^2(\Omega)\) is compact, it follows that
\begin{equation}\label{5.12}
\|(u_n,v_n)\|^2_{\mathbb{H}} - 2_s^* \int_{\Omega} |u_n|^{\alpha} |v_n|^{\beta} \, \dx = o(1),
\end{equation}
using the definition of $\widetilde{S}_{(\alpha,\beta)}(\Omega)$, we obtain
\begin{align*}
    \|(u_n,v_n)\|^2_{\mathbb{H}} \geq \widetilde{S}_{(\alpha,\beta)}(\Omega) \left( \int_{\Omega} |u_n|^{\alpha} |v_n|^{\beta} \, \dx\right)^{\frac{2}{2_s^*}}.
\end{align*}
Hence
\[
\|(u_n,v_n)\|^2_{\mathbb{H}} - 2_s^* \left( \widetilde{S}_{(\alpha,\beta)}(\Omega)\right)^{- \frac{2_s^*}{2}} \|(u_n,v_n)\|^{2_s^*}_{\mathbb{H}} \leq o(1),
\]
thus
\[
\|(u_n,v_n)\|^2_{\mathbb{H}} \left( 1 - 2_s^* \left( \widetilde{S}_{(\alpha,\beta)}(\Omega)\right)^{- \frac{2_s^*}{2}}\|(u_n,v_n)\|^{2_s^* - 2}_{\mathbb{H}} \right) \leq o(1).
\]
If \(\|(u_n,v_n)\|_{\mathbb{H}} \to 0\), passing in the limit in \eqref{5.9} we reach a contradiction with the assumption $c>0$. Therefore, we deduce
\[
\left( 1 - 2_s^* \left( \widetilde{S}_{(\alpha,\beta)}(\Omega)\right)^{- \frac{2_s^*}{2}} \|(u_n,v_n)\|^{2_s^* - 2}_{\mathbb{H}} \right) \leq 0
\]
i.e.,
\[
\|(u_n,v_n)\|^2_{\mathbb{H}} \geq {2_s^*}^{-\frac{2}{2_s^*-2}} \left( \widetilde{S}_{(\alpha,\beta)}(\Omega)\right)^{ \frac{2_s^*}{2_s^*-2}}.
\]
From \eqref{5.12}, we have that
\begin{align*}
J(u_n,v_n) &= \left( \frac{1}{2} - \frac{1}{2_s^*} \right) \|(u_n,v_n)\|^2_{\mathbb{H}} + \frac{1}{2_s^*} \left( \|(u_n,v_n)\|^2_{\mathbb{H}} - 2_s^* \int_{\Omega} |u_n|^{\alpha} |v_n|^{\beta} dx + o(1) \right) \\
&= \frac{s}{N} \|(u_n,v_n)\|^2_{\mathbb{H}} + o(1) \\
&\geq \frac{s}{N} \left({2_s^*}^{-\frac{2}{2_s^*-2}} \left( \widetilde{S}_{(\alpha,\beta)}(\Omega)\right)^{ \frac{2_s^*}{2_s^*-2}} \right) + o(1),
\end{align*}
which contradicts the fact that \(c < \frac{s}{N} {2_s^*}^{-\frac{N-2s}{2s}} \left( \widetilde{S}_{(\alpha,\beta)}(\Omega)\right)^{\frac{N}{2s}}\). Thus \((u,v) \neq (0,0)\).
\end{proof}

We are now in position to state and prove the main result of this work
\begin{theorem}
Let $\Omega \subset \R^N, {N> 8s}$, be a bounded domain with smooth boundary. Then, the system \eqref{main problem} has a non-trivial solution provided $\Lambda_k \leq \mu_1 \leq \mu_2< \Lambda_{k+1}$ and $(\mathcal{H})$ holds.
\end{theorem}
\begin{proof}
From Lemma \ref{geometry-1} and Lemma \ref{geometry-2}, the functional $J$ satisfies the geometric conditions of the Rabinowitz Linking Theorem. From Lemma \ref{energy-threshold-estimate} and Lemma \ref{boundedness}, it is easy to verify that the $\rm (PS)_c$-condition holds for the functional $J$ below the critical threshold $c_0$ as given in Lemma \ref{boundedness}. Thus, the functional $J$ possesses at least one non-trivial critical point which is a solution to our system \eqref{main problem}.
\end{proof}

\section*{Acknowledgment}
\noindent R. Kumar acknowledges TIFR-Centre for Applicable Mathematics, India for the Institute Postdoctoral Fellowship to conduct this research.\\
A. Ortega is partially funded by Vicerrectorado de Investigación,
Transferencia del Conocimiento y Divulgación Científica of Universidad Nacional de Educación a Distancia under research project Talento Joven UNED 2025, Ref: 2025/00151/001.

%%%%%%%%%%%%%%%%%%%%%%%%%%%%%%%%%%%%%%%%%%
%\bibliography{Ref}

\begin{thebibliography}{99}

\bibitem{Akhmediev2}
N.~Akhmediev and A.~Ankiewicz.
\newblock {\em Solitons, Nonlinear Pulses and Beams}.
\newblock Champman \& Hall, London, 1997.

\bibitem{Akhmediev}
N.~Akhmediev and A.~Ankiewicz.
\newblock Partially coherent solitons on a finite background.
\newblock {\em Phys. Rev. Lett.}, 82, 1999.

\bibitem{Alves-2008}
C.~Alves.
\newblock Multiplicity of solutions for a mixed boundary elliptic system.
\newblock {\em Rocky Mountain J. Math.}, 38, 2008.

\bibitem{alves-2000}
C.~O. Alves, D.~C. de~Morais~Filho, and M.~A.~S. Souto.
\newblock On systems of elliptic equations involving subcritical or critical {S}obolev exponents.
\newblock {\em Nonlinear Anal.}, 42(5):771--787, 2000.

\bibitem{Barrios-2012}
B.~Barrios, E.~Colorado, A.~de~Pablo, and U.~Sánchez.
\newblock On some critical problems for the fractional {L}aplacian operator.
\newblock {\em J. Differential Equations}, 252(11), 2012.

\bibitem{Bouchekif-2008}
M.~Bouchekif and Y.~Nasri.
\newblock On elliptic system involving critical {S}obolev–{H}ardy exponents.
\newblock {\em Mediterr. J. Math.}, 5, 2008.

\bibitem{Brezis-1983}
H.~Brezis and L.~Nirenberg.
\newblock Positive solutions of nonlinear elliptic equations involving critical {S}obolev exponents.
\newblock {\em Comm. Pure Appl. Math.}, 36(4), 1983.

\bibitem{Colorado-2013}
C.~Brändle, E.~Colorado, A.~de~Pablo, and U.~Sánchez.
\newblock A concave-convex elliptic problem involving the fractional {L}aplacian.
\newblock {\em Proc. Roy. Soc. Edinburgh Sect. A}, 143(1), 2013.

\bibitem{Cabre-2010}
X.~Cabré and J.~Tan.
\newblock Positive solutions of nonlinear problems involving the square root of the {L}aplacian.
\newblock {\em Adv. Math.}, 224(5), 2010.

\bibitem{Silvestre-2007}
L.~Caffarelli and L.~Silvestre.
\newblock An extension problem related to the fractional {L}aplacian.
\newblock {\em Comm. Partial Differential Equations}, 32(7-9):1245--1260, 2007.

\bibitem{Capella-2011}
A.~Capella, J.~Dávila, L.~Dupaigne, and Y.~Sire.
\newblock Regularity of radial extremal solutions for some non-local semilinear equations.
\newblock {\em Comm. Partial Differential Equations}, 36(8), 2011.

\bibitem{Capozzi-1985}
A.~Capozzi, D.~Fortunato, and G.~Palmieri.
\newblock An existence result for nonlinear elliptic problems involving critical {S}obolev exponent.
\newblock {\em Ann. Inst. H. Poincaré Anal. Non Linéaire}, 2, 1985.

\bibitem{Ortega-2021-Reg}
J.~Carmona, E.~Colorado, T.~Leonori, and A.~Ortega.
\newblock Regularity of solutions to a fractional elliptic problem with mixed {D}irichlet-{N}eumann boundary data.
\newblock {\em Adv. Calc. Var.}, 14(4), 2021.

\bibitem{Chabrowski-2007}
J.~Chabrowski and B.~Ruf.
\newblock On the critical {N}eumann problem with lower order perturbations.
\newblock {\em Colloq. Math.}, 108(2), 2007.

\bibitem{Chen-2012}
Z.~Chen and W.~Zou.
\newblock Positive least energy solutions and phase separation for coupled {S}chrödinger equations with critical exponent.
\newblock {\em Arch. Ration. Mech. Anal.}, 205(2), 2012.

\bibitem{Ortega-2019}
E.~Colorado and A.~Ortega.
\newblock The {B}rezis-{N}irenberg problem for the fractional {L}aplacian with mixed {D}irichlet-{N}eumann boundary conditions.
\newblock {\em J. Math. Anal. Appl.}, 473(2):1002--1025, 2019.

\bibitem{Colorado-2003}
E.~Colorado and I.~Peral.
\newblock Semilinear elliptic problems with mixed {D}irichlet-{N}eumann boundary conditions.
\newblock {\em J. Funct. Anal.}, 199(2):468--507, 2003.

\bibitem{Morais-2012}
D.~C. de~Morais~Filho, L.~F.~O. Faria, O.~H. Miyagaki, and F.~R. Pereira.
\newblock One sided resonance for a mixed boundary elliptic system involving critical {S}obolev exponent.
\newblock {\em Houston J. Math.}, 38(3):915--931, 2012.

\bibitem{Esry}
C.~Esry, J.~Greene, B.~Burke~Jr., and J.~Bohn.
\newblock Hartree–fock theory for double condensates.
\newblock {\em Phys. Rev. Lett.}, 78, 1997.

\bibitem{Frantzeskakis}
D.~Frantzeskakis.
\newblock Dark solitons in atomic {B}ose–{E}instein condensates: from theory to experiments.
\newblock {\em J. Phys. A Math. Theory}, 43, 2010.

\bibitem{Gazzola-1997}
F.~Gazzola and B.~Ruf.
\newblock Lower order perturbatios of critical growth nonlinearities in semilinear elliptic equations.
\newblock {\em Advances in Differential Equations}, 2(4), 1997.

\bibitem{Hojo}
H.~Hojo, H.~Ikezi, K.~Mima, and K.~Nishikawa.
\newblock Coupled nonlinear electron-plasma and ion-acoustic waves.
\newblock {\em Phys. Rev. Lett.}, 33, 1974.

\bibitem{Hsu-2011}
T.~Hsu and H.~Li.
\newblock Multiplicity of positive solutions for singular elliptic systems with critical {S}obolev-{H}ardy and concave exponents.
\newblock {\em Acta Math. Sci.}, 31B(3), 2011.

\bibitem{Kakutani}
T.~Kakutani, T.~Kawahara, and N.~Sugimoto.
\newblock Nonlinear interaction between short and long capillary-gravity waves.
\newblock {\em J. Phys. Soc. Japan}, 39, 1975.

\bibitem{Kaminow}
I.~Kaminow.
\newblock Polarization in optical fibers.
\newblock {\em IEEE J. Quantum Electron.}, 17, 1981.

\bibitem{Kang-2015}
D.~Kang, J.~Luo, and X.~Shi.
\newblock Solutions to elliptic systems involving doubly critical nonlinearities and {H}ardy-type potentials.
\newblock {\em Acta Math. Sci.}, 35B(2), 2015.

\bibitem{Kivshar}
Y.~Kivshar and B.~Luther-Davies.
\newblock Dark optical solitons: physics and applications.
\newblock {\em Phys. Rep.}, 298, 1998.

\bibitem{Lions-Magenes}
J.-L. Lions and E.~Magenes.
\newblock {\em Non-homogeneous boundary value problems and applications. {V}ol. {I}}, volume Band 181 of {\em Die Grundlehren der mathematischen Wissenschaften}.
\newblock Springer-Verlag, New York-Heidelberg, 1972.
\newblock Translated from the French by P. Kenneth.

\bibitem{Lions-1988}
P.-L. Lions, F.~Pacella, and M.~Tricarico.
\newblock Best constants in {S}obolev inequalities for functions vanishing on some part of the boundary and related questions.
\newblock {\em Indiana Univ. Math. J.}, 37(2):301--324, 1988.

\bibitem{Manakov}
S.~Manakov.
\newblock On the theory of two-dimensional stationary self-focusing of electromagnetic waves.
\newblock {\em Sov. Phys.–JETP}, 38, 1974.

\bibitem{Menyuk}
C.~Menyuk.
\newblock Nonlinear pulse propagation in birefringent optical fibers.
\newblock {\em IEEE J. Quantum Electr.}, 23, 1981.

\bibitem{Ortega-2024-Subcritical}
G.~Molica~Bisci, A.~Ortega, and L.~Vilasi.
\newblock Subcritical nonlocal problems with mixed boundary conditions.
\newblock {\em Bull. Math. Sci.}, 14(1):Paper No. 2350011, 23, 2024.

\bibitem{Ortega2023}
A.~Ortega.
\newblock Concave-convex critical problems for the spectral fractional laplacian with mixed boundary conditions.
\newblock {\em Fract. Calc. Appl. Anal.}, 26:305--335, 2023.

\bibitem{Pohozaev-1965}
S.~I. Poho\v{z}aev.
\newblock On the eigenfunctions of the equation {$\Delta u+\lambda f(u)=0$}.
\newblock {\em Dokl. Akad. Nauk SSSR}, 165:36--39, 1965.

\bibitem{Rabinowitz-1986-book}
P.~H. Rabinowitz.
\newblock {\em Minimax methods in critical point theory with applications to differential equations}, volume~65 of {\em CBMS Regional Conference Series in Mathematics}.
\newblock Conference Board of the Mathematical Sciences, Washington, DC, 1986.

\bibitem{Ruegg}
C.~Rüegg et~al.
\newblock Bose-einstein condensation of the triple states in the magnetic insulator $ticuci_3$.
\newblock {\em Nature}, 423, 2003.

\bibitem{Servadei-2013}
R.~Servadei.
\newblock The {Y}amabe equation in a non-local setting.
\newblock {\em Adv. Nonlinear Anal.}, 2, 2013.

\bibitem{Servadei-2014}
R.~Servadei.
\newblock A critical fractional {L}aplace equation in the resonant case.
\newblock {\em Topological Methods in Nonlinear Analysis}, 43, 2014.

\bibitem{Servadei-2015}
R.~Servadei and E.~Valdinoci.
\newblock {F}ractional {L}aplacian equations with critical {S}obolev exponent.
\newblock {\em Rev. Mat. Complut.}, 28, 2015.

\bibitem{Tavares-2020}
H.~Tavares and S.~You.
\newblock Existence of least energy positive solutions to {S}chrödinger systems with mixed competition and cooperation terms: the critical case.
\newblock {\em Cal. Var.}, 59, 2020.

\bibitem{Tavares-2022}
H.~Tavares, S.~You, and W.~Zou.
\newblock Least energy positive solutions of critical {S}chrödinger systems with mixed competition and cooperation terms: the higher dimensional case.
\newblock {\em J. Funct. Anal.}, 283, 2022.

\bibitem{Zhang-1989}
D.~Zhang.
\newblock On multiple solutions of {$\Delta u+\lambda u+|u|^{\frac{4}{n-2}}=0$}.
\newblock {\em Nonlinear Anal.}, 13, 1989.

\end{thebibliography}
%\bibliographystyle{abbrv}

\end{document}